\documentclass[12pt,reqno]{amsart}

\usepackage{amssymb}
\usepackage{amsthm}
\usepackage{amsmath}
\usepackage{amsaddr}
\usepackage{bm}
\usepackage{cite}
\usepackage{mathrsfs}
\usepackage{xcolor}

\usepackage[OT1]{fontenc}
\RequirePackage{times}

\usepackage[left=2.1cm, right=2.1cm, top=3cm]{geometry}

\allowdisplaybreaks
\newtheorem{theorem}{Theorem}[section]

\newtheorem{lemma}[theorem]{Lemma}

\newtheorem{remark}[theorem]{Remark}

\def\vp{\varphi}
\def\fhi{\varphi}

\def\n{\textbf{\textit{n}}}
\def\R3{\mathbb{R}^3}
\def\F2o{\overline{F_2}}

\def \E{E}

\def \hsu{\overline{h}}
\def \hgiu{\underline{h}}
\def\d{{\rm d}}
\def \l {\langle}
\def \r {\rangle}

\def\u{\textbf{\textit{u}}}
\def\uu{\textbf{\textit{u}}}

\def\ddt{\frac{\d}{\d t}}

\def\vp{\varphi}

\def\ov{\overline}

\def \dt{\partial_t}

\def\div{\mathrm{div}\, }

\def \R{\mathbb{R}}

\DeclareMathOperator{\sign}{sign}

\def \au {\rm}
\def \ti {\it}
\def \jou {\rm}
\def \bk {\it}
\def \no#1#2#3 {{\bf #1} (#3), #2.}
\def \eds#1#2#3 {#1, #2, #3.}
\def \nome#1#2 {{\bf #1}, (#2).}

\begin{document}

\title[Cahn-Hilliard-Darcy system with mass source]
{On the Existence of Strong Solutions to \\ the Cahn-Hilliard-Darcy system with mass source}

\author[Giorgini-Lam-Rocca-Schimperna]{Andrea Giorgini$^\dagger$, Kei Fong Lam$^\ddagger$, Elisabetta Rocca$^\ast$ \& Giulio Schimperna$^\ast$}
\address{$^\dagger$Department of Mathematics, Indiana University\\
Bloomington, IN 47405, USA\\
agiorgin@iu.edu}

\address{$^\ddagger$Department of Mathematics, Hong Kong Baptist University\\ 
Kowloon Tong, Hong Kong\\
akflam@hkbu.edu.hk}

\address{$^\ast$Dipartimento di Matematica ``F. Casorati", Universit\`{a}
degli Studi di Pavia and IMATI - C.N.R. \\
via Ferrata 5, 27100, Pavia, Italy \\
elisabetta.rocca@unipv.it, giusch04@unipv.it}

\date{\today}



\begin{abstract}
We study a diffuse interface model describing the evolution of the flow of
a binary fluid in a Hele-Shaw cell. The model consists of a Cahn--Hilliard--Darcy (CHD) type system with transport and mass source. 
A relevant physical application is related to tumor growth dynamics, which in particular justifies the occurrence of a mass inflow. We 
study the initial-boundary value problem for this model and prove global existence and uniqueness of strong solutions in two space dimensions as well as local
existence in three space dimensions.
\end{abstract}

\maketitle


\section{Introduction}

In a bounded domain $\Omega\subset \mathbb{R}^d$, $d=2,3$, with smooth boundary $\partial \Omega$,
we consider the initial-boundary value problem
\begin{equation}
\begin{cases}
\label{CHHS}
\u  +\nabla q = -\vp \nabla \mu,\\
\mathrm{div}\, \u=S, \\
\partial_t \vp +\div ( \u \vp )=   \Delta \mu+ S,\\
\mu= - \Delta \vp +  \Psi'(\vp),
\end{cases}
 \quad \text{ in } \Omega\times (0,T).
\end{equation}
Here, $\vp$ denotes the difference of the fluid concentrations, $\uu$  is the fluid velocity,
$q$ is the pressure, $\mu$ is the chemical potential and 
\begin{equation}
\label{S}
S= -m \vp +h(\vp),
\end{equation}
where $m$ is a positive constant and $h:[-1,1]\rightarrow [-1,1]$. 
The choice of the mass source term $S$ is dictated by applications of the above system to the tumor growth dynamics.

The CHD system \eqref{CHHS} can be viewed as the simplest version of some recently introduced diffuse 
interface models for tumor growth (cf., e.g., \cite{CLLW2009,Cristini,Frieboes,GLNS2017,GLSS2016,Hawkins,Oden,Wise}).
In this class of models, variants of system \eqref{CHHS} are often coupled with other relations
describing the evolution of additional variables (e.g., a nutrient or a drug), generally obeying
reaction-diffusion type equations. 
Analytical results related to well-posedness, singular limits and long-time behavior of 
models in this class have been established in
\cite{CGH,CGRSAA,CGRSVV,FGR,FLR,GLDirichlet,GLDarcy,GLNeumann,RoccaScala}
for tumor growth models based on the coupling of Cahn--Hilliard (for the tumor density) and reaction-diffusion (for the nutrient or other chemical factors) equations, and in \cite{GLDarcy,G2020,GGW2018,JWZ,LTZ,Melchionna} for models of Cahn--Hilliard--Darcy type. 
We also mention the works \cite{BCG,CG2020,EG2019,EG2019-2,EL2021}  dealing with models of Cahn--Hilliard--Brinkman type, which are characterized by an additional viscosity term occurring in the left-hand side of the velocity equation. 

In the CHD system \eqref{CHHS}, the $\u$-equation is obtained from a generalized form of the Darcy law, where the constant in front of the term 
measuring the excess adhesion force at the diffusive tumor/host tissue interfaces has been set equal to one for simplicity.  
The $\vp$-equation is a convective Cahn--Hilliard type equation, which is derived from the mass balance. The mass source term $S$ accounts for cell
proliferation (or the rate of change in tumor volume, see \cite{Frieboes, Wise}). Then, the 
chemical potential $\mu$ represents the variational derivative of the usual Ginzburg-Landau
free energy functional
 \begin{equation}
 \label{FE}
 \E(\vp)= \int_\Omega  \frac12 |\nabla \vp|^2
+ \Psi(\vp) \, \d x
 \end{equation}
 in which the function $\Psi$ represents the physically relevant Flory--Huggins logarithmic potential (cf.~\cite{Dai, FLRS2018}):
\begin{equation}
\label{LOG}
\Psi(s)=\frac{\theta}{2}\big[ (1+s)\log(1+s)+(1-s)\log(1-s)\big]-\frac{\theta_0}{2} s^2,
 \quad s \in [-1,1],
\end{equation}
where $\theta$ and $\theta_0$ are two positive constant satisfying $0<\theta<\theta_0$.
 
The expression \eqref{S} for the mass source $S$  arises from the specific application 
of system \eqref{CHHS} to tumor evolution problems suggested by the 
recent work \cite{FLRS2018}, where a multicomponent variant of our system is considered,
coupled with a further relation describing the evolution of a nutrient. In \cite{FLRS2018} a global-in-time existence result for {\it weak solutions}\/ has been obtained in the 
three dimensional case; however, the questions of additional regularity and uniqueness are left as open issues. 

When $S$ has a linear dependence on $\vp$ and we neglect the velocity, the $\vp$-equation reduces to the so-called Cahn--Hilliard--Oono equation,  that accounts for long-range (nonlocal) interactions 
in the phase separation process (cf.~\cite{OP}): we refer to \cite{CMZ,GGM2017} (see also the references therein)
for some mathematical results. 
In the case when no source terms occur, i.e., $S= 0$, \eqref{CHHS} is referred to as the Cahn--Hilliard--Hele-Shaw (CHHS)
system and is used to describe two-phase flows in the Hele-Shaw geometry (cf.~\cite{LLG2002}). The CHHS system with zero 
mass source term has been studied by several authors, both numerically and analytically: existence of global weak 
solutions in three dimensions has been shown in \cite{DGL,FengWise,GGW2018}; local and global existence and uniqueness of strong 
solutions in two and three dimensions, respectively, have been achieved in \cite{G2020,GGW2018,LTZ,WZ13}; finally, long-time behavior and 
stability of local minimizers have been addressed in \cite{GGW2018,LTZ,WW12}.
Meanwhile, in the case where $S$ is prescribed, global weak existence and local strong well-posedness for system \eqref{CHHS} has been
shown in \cite{JWZ}, and a related optimal control problem has been studied in \cite{SW2018}.

The CHHS system can be also formally viewed as an appropriate limit
of the classical Navier--Stokes--Cahn--Hilliard (NSCH) system 
(cf.~\cite{AMW,DGL,LLG2002,HH}).
This system, together with its many variants, is a fundamental model in the
evolutionary theory of binary fluids and, as such, is extensively studied in the mathematical literature,
with most of the results referring to local-in-time well-posedness in three space dimensions 
and global-in-time well-posedness in two dimensions under various 
assumptions (see e.g.~\cite{A2009,GMT2019}). The mathematical difficulties associated with the
CHHS system are similar to those associated with the NSCH equations: one gains an algebraic relation
by dropping the nonlinear advection term in the $\u$-equation but then, at the same time, one loses the regularizing viscosity term. 

The {\it multi-phase}\/ variants of Cahn--Hilliard--Darcy and related systems are 
comparatively less studied in the literature, especially when they refer to tumor-growth
models (in such a situation, one may distinguish between the proliferating and necrotic tumor cells, which is what gives the multicomponent evolution).
In \cite{Dai} and \cite{FLRS2018} a simplification of the tumor model introduced in \cite{Chen} is studied, and the existence 
of a weak solution is proved in the three dimensional case. In \cite{GLNS2017}, a vectorial Cahn--Hilliard--Darcy model 
is proposed to describe multi-phase tumor evolution by means of a volume-average velocity (so satisfying
a much simpler relation). 
Furthermore, differently from the system studied in \cite{Dai, FLRS2018} which consists of a single Cahn--Hilliard equation ruling the evolution of the total tumor volume fraction coupled with transport-type
equations for the individual tumor species, 
%
in the model of \cite{GLNS2017} each tumor species is governed by a distinct Cahn--Hilliard-type equation, so that the corresponding natural energy identity permits to deduce better a priori estimates.

System  \eqref{CHHS} is complemented with the following boundary and initial conditions:
\begin{equation}
\label{bc}
\begin{cases}
\partial_\n \vp=0, \quad q=0, \quad \partial_\n \mu-(\uu \cdot \n)\vp=0, \quad &\text{ on } \partial \Omega \times (0,T),\\
\vp( \cdot,0)=\vp_0, \quad &\text{ in } \Omega,
\end{cases}
\end{equation}
where $\n$ is the unit outward normal vector to the boundary $\partial \Omega$. This choice is mainly
motivated by the specific application to tumor evolution addressed, in the multi-phase setting, in \cite{FLRS2018}
with boundary conditions analogous to the above ones. Physically speaking, conditions \eqref{bc}
prescribe that the mass inflow or outflow from the reference domain $\Omega$ depends on the forcing term $S$ only,
and it is not otherwise influenced by the boundary behavior of the macroscopic velocity $\uu$. 
It is worth noting that the (easier) case of Dirichlet boundary 
conditions for $\mu$ could also be treated, but we preferred to handle \eqref{bc} 
which seems to be more reasonable from the modeling point of view (cf.~\cite{FLRS2018} for 
additional considerations). On the contrary, the equally meaningful case of no-flux conditions for 
$\mu$ appears mathematically more difficult due to lower coercivity. 

In this paper we restrict our analysis to system \eqref{CHHS}-\eqref{bc}. 
This represents the one species case when the nutrient evolution is not taken into account. Our main goal is to prove the existence of {\it strong solutions} in two and three 
dimensions and uniqueness in two dimensions for system \eqref{CHHS}. This is indeed an open problem for the whole system 
accounting also for the nutrient evolution (cf.~\cite{FLRS2018}), especially in the multi-component case when only existence of global {\it weak solutions} is known. 

From the mathematical viewpoint, a main difficulty in the study of the regularity problem for CHD systems lies in the strong coupling that 
originates from the transport term. In particular, the well-known regularity result for the Cahn--Hilliard equation with convection in 
\cite[Lemma 3]{A2009} cannot be applied since the only available property from the energy equation on the velocity is $\uu \in L^2(\Omega \times (0,T))$. 
In the recent literature, this issue has been overcome in two ways: in \cite{GGW2018} the transport term has been rewritten by applying the Leray 
projection to the Darcy law, whereas in \cite{G2020} the vorticity equation ($\mathrm{curl}\, \uu= \nabla \vp\cdot \nabla \mu^{\perp}$) has been 
used to derive an estimate of the velocity in $H^1(\Omega)$. Both approaches follow the idea of eliminating the pressure. However, such arguments 
fail in the case of the system \eqref{CHHS}-\eqref{bc} because of the non vanishing divergence of $\uu$ and the lack of a boundary condition for
$\uu$ ($\vp$ might vanish on the boundary), which does not permit us to recover a full gradient estimate of $\uu$ from the $\mathrm{curl}\, \uu$ and $\div \uu$. 
Therefore, we use a different strategy which relies on the algebraic structure of the system. We indeed notice that, by inserting the Darcy 
law \eqref{CHHS}$_1$ into the Cahn--Hilliard equation \eqref{CHHS}$_3$, we can formulate some elliptic problems for the pressure $q$ (see \eqref{PP2}) 
and for the function $\frac{\vp}{1+\vp^2}q$ (see \eqref{phiqP}), which are characterized by lower order terms on the right-hand side. These observations are 
fundamental to study the corresponding elliptic problem for the chemical potential $\mu$ (see \eqref{muP}) and to deduce a $H^2$-estimate of $\mu$ only in terms of the $L^2$ norms of $\nabla \mu$ and $\uu$, and Lipschitz norms of $\vp$.
It is finally worth noting that the technique we use to obtain the additional regularity estimates
is independent of the specific form of the forcing term $S$ and may be applied to different, or more general,
models of Cahn--Hilliard--Darcy type with or without logarithmic potential, the former case being as usual the more 
complex one. In other words, our approach seems to provide a new and rather general strategy 
to address a wide class of Cahn--Hilliard--Darcy systems with mass source and is not restricted to the 
specific case of tumor-growth models.

The remainder of the paper is organized as follows: in Section~\ref{S2} we report some mathematical tools which will be used in the 
analytical proofs; then, Section~\ref{pre} lists the assumptions and outlines the key idea of reformulating the elliptic systems for the pressure and the chemical potentials.  Section~\ref{S3} is devoted to the global well-posedness result in two dimensions, while Section~\ref{S4} deals with the local existence result in three dimensions.


\section{Mathematical Setting}
\label{S2}
\setcounter{equation}{0}


\subsection{Function spaces}
Let $X$ be a (real) Banach or Hilbert space, whose norm is denoted
by $\|\cdot\|_X$. The space $X'$ indicates the
dual space of $X$ and $\l \cdot,\cdot\r$ denotes the duality product. 
In a bounded domain $\Omega \subset \mathbb{R}^d$ with smooth boundary $\partial \Omega$,
$W^{k,p}(\Omega)$, $k\in \mathbb{N}$ and
$p\in [1,+\infty]$, are the Sobolev spaces of real measurable functions on $\Omega$. We
denote by $H^k(\Omega)$ the Hilbert space
$W^{k,2}(\Omega)$ and by $\|\cdot\|_{H^k(\Omega)}$ its norm.
In particular, $H=L^2(\Omega)$ with inner product and norm denoted by $(\cdot,\cdot)$ and $\|\cdot\|$, respectively. The space $H^{1}(\Omega)$ is endowed with the norm $\|f\|_{H^1(\Omega)}^2=\|\nabla f \|^2+\|f\|^2$. The norm of the dual space $(H^1(\Omega))'$ is denoted by $\| \cdot \|_{\ast}$.
For every $f\in (H^1(\Omega))'$, we denote by $\overline{f}$ the total mass of $f$ defined by $\overline{f}=\frac{1}{|\Omega|}\l f,1\r$.
We recall the following Poincar\'{e}'s inequality
\begin{equation}
\label{poincare}
\|f-\overline{f}\|\leq C\|\nabla f\|,\quad \forall\,
f\in H^1(\Omega),
\end{equation}
where the constant $C$ depends only on $d$ and $\Omega$.
\smallskip


\subsection{Interpolation and product inequalities.}
\label{interpolations} 

We recall here some well-known interpolation inequalities
in Sobolev spaces which can be found in classical
literature (see e.g.~\cite{Engler,Te01}):
\begin{itemize}
\item[$\diamond$] Ladyzhenskaya's  inequalities
\begin{align}
\label{L}
\|f\|_{L^4(\Omega)}
\leq C  \| f\|^{\frac12} \|f\|_{H^1(\Omega)}^{\frac12}, \quad &&\forall
\, f \in H^1(\Omega), \  d=2,\\
\label{L3}
\|f\|_{L^4(\Omega)}
\leq C  \| f\|^{\frac14} \|f\|_{H^1(\Omega)}^{\frac34}, \quad &&\forall
\, f \in H^1(\Omega), \  d=3.
\end{align}
\item[$\diamond$] Agmon's inequalities
\begin{align}
&\|f\|_{L^\infty(\Omega)}
\leq C\|f\|^\frac12\|f\|_{H^2(\Omega)}^\frac12,
\quad &&\forall \, f \in H^2(\Omega), \  d=2,\label{Ad2}\\
&\|f\|_{L^\infty(\Omega)}
 \leq C\|f\|_{H^1(\Omega)}^\frac12\|f\|_{H^2(\Omega)}^\frac12,
 \quad &&\forall \, f \in H^2(\Omega), \  d=3. \label{Ad3}
\end{align}
\item[$\diamond$] Br\'{e}zis--Gallouet--Wainger inequality
\begin{equation} 
\label{BWd2}
\| f\|_{L^\infty(\Omega)}\leq C 
\| f\|_{H^1(\Omega)}
\bigg( 1+ \log^{\frac12} \Big(1 +\frac{\|f\|_{W^{1,q}(\Omega)}}{\| f\|_{H^1(\Omega)}}\Big) \bigg), \quad \forall \, f \in W^{1,q}(\Omega), \ q>2, \  d=2. 
\end{equation}

\item[$\diamond$] Gagliardo--Nirenberg inequalities
\begin{align}
\label{GN2}
&\| f\|_{L^p(\Omega)}\leq C \| f\|_{L^q(\Omega)}^{1-\theta} \| f\|_{H^1(\Omega)}^\theta, \quad &&\forall \, f \in H^1(\Omega), 1\leq q\leq p<\infty, \ \theta= 1-\frac{q}{p}, \ d=2,\\
\label{GN3}
&\| f\|_{L^\infty(\Omega)}\leq C \| f\|^{1-\theta} \| f\|_{W^{1,q}(\Omega)}^\theta, \quad &&\forall \, f \in W^{1,q}(\Omega), \ q>3, \ \theta= \frac{3q}{5q-6}, \ d=3.
\end{align}
\end{itemize}

\subsection{Generalized Gronwall lemma.}
\label{Gronwall} 
We report a generalized Gronwall type lemma (see e.g.~\cite{G2020}). 

\begin{lemma}
\label{GL2}
Let $f$ be a positive absolutely continuous function on $[0,T]$ and $g$, $h$ two summable functions on $[0,T]$ which satisfy the differential inequality
$$
\ddt f(t)\leq g(t)f(t)\log\big( C+ f(t)\big)+h(t)
$$
for almost every $t\in [0,T]$ and for some $C>1$. Then, we have
$$
f(t)\leq  \big( C+f(0)\big)^{\mathrm{e}^{\int_0^t g(\tau)\, \d \tau}} \mathrm{e}^
{ \int_0^t \mathrm{e}^{\int_\tau^t g(s)\, \d s} h(\tau) \, \d \tau},
\quad \forall 
\, t \in [0,T].
$$
\end{lemma}
\smallskip

\section{Assumptions and Elliptic Structure of the System}
\label{pre}
\setcounter{equation}{0}

We present our basic assumptions, which will be kept for the two dimensional and for the three dimensional case:
\begin{itemize}
\item[(A1)] the function $S$ is given by \eqref{S}, where $m$ is a positive constant and 
$h:[-1,1]\rightarrow [-1,1]$ is of class $\mathcal{C}^2([-1,1])$. 
More precisely, setting $\hgiu:=\min_{s\in[-1,1]} h(s)$ and $\hsu:=\max_{s\in[-1,1]} h(s)$,
we suppose that
\begin{equation}
\label{hp:h}
  -1 < \frac{\hgiu}{|\Omega|m}, \qquad 
  \frac{\hsu}{|\Omega|m} < 1;   
\end{equation}
\item[(A2)] the function $\Psi$ is the logarithmic potential defined in \eqref{LOG};

\item[(A3)] the initial datum $\fhi_0$ lies in $H^2(\Omega)\cap L^\infty(\Omega)$ 
with $\| \vp_0\|_{L^\infty(\Omega)}\leq 1$ and $\partial_{\n} \fhi_0 = 0$ 
on $\partial\Omega$. Moreover, the following additional condition holds:
\begin{equation}
\label{fhi01}
  \mu_0:= - \Delta \fhi_0 + \Psi'(\fhi_0) \in H^1(\Omega).
\end{equation}
\end{itemize}

\begin{remark}\label{rem:h}
Notice that assumption \eqref{hp:h} is required in order to prove that the total tumor mass during the evolution remains strictly 
in between the critical values $-1$ and $1$ (cf.~\eqref{mass}). 
\end{remark}

\begin{remark}\label{rem:H}
 By Sobolev's embeddings and standard regularity results for elliptic problems with maximal monotone perturbations 
 (see, e.g., \cite[Lemma 7.4]{GGW2018}), \eqref{fhi01} entails that $\fhi_0 \in W^{2,6}(\Omega)$ if $d=3$ and $\fhi_0 \in W^{2,p}(\Omega)$ if $d=2$ for any $p\in [2,\infty)$. Then, defining
 $q_0$ as the solution of the elliptic problem
 \begin{align}
  \label{q01}
  \begin{cases}
 - \Delta q_0 = - m \fhi_0 + h(\fhi_0) + \div(\fhi_0 \nabla \mu_0) 
       \quad &\text{ in } \Omega,\\
 q_0 = 0 \quad &\text{ on } \partial\Omega,
     \end{cases}
 \end{align}
the elliptic regularity permits us to check that $q_0\in H^1_0(\Omega)$.
Thus, setting also
 \begin{equation}
 \label{u0}
   \uu_0 := -\nabla q_0 - \fhi_0 \nabla \mu_0,
 \end{equation}
 we also have that $\uu_0\in L^2(\Omega)$.
\end{remark}

Next, we discuss the {\it elliptic} structure of the system \eqref{CHHS}-\eqref{bc}. First, we consider the problem for the pressure $q$ obtained from \eqref{CHHS}$_1$-\eqref{CHHS}$_2$
\begin{equation}
\label{Smu2-0}
\begin{cases}
-\Delta q= S+ \div (\vp \nabla \mu), \quad &\text{in } \ \Omega,\\
q=0,\quad &\text{on }\ \partial \Omega.
\end{cases}
\end{equation}
For regularity purposes that will be developed in the Sections \ref{S3} and \ref{S4}, we reformulate the Poisson equation \eqref{Smu2-0}$_1$ with the aim to obtain a right-hand side containing only one spatial derivative that acts on the chemical potential $\mu$. 
We start by substituting the relation of the velocity in \eqref{CHHS}$_3$ which leads to the expression
$$
\partial_t \vp + \div \big( (-\nabla q-\vp \nabla \mu)\vp\big)=\Delta \mu +S,
$$
which is equivalent to
$$
\div \big( (1+\vp^2)\nabla \mu \big)= \partial_t \vp - \div(\vp \nabla q) -S.
$$
Then, using the above relation, we write
\begin{align*}
\div(\vp \nabla \mu)&= \div \Big( \frac{\vp}{1+\vp^2} (1+\vp^2)\nabla \mu\Big)\\
&=(1+\vp^2) \nabla \mu \cdot \nabla \Big( \frac{\vp}{1+\vp^2} \Big) + \frac{\vp}{1+\vp^2}\div \big((1+\vp^2)\nabla \mu\big)\\
&=(1+\vp^2) \nabla \mu \cdot \nabla \Big( \frac{\vp}{1+\vp^2} \Big) + \frac{\vp}{1+\vp^2} \big( \partial_t \vp -S\big)- \frac{\vp}{1+\vp^2}
\div(\vp \nabla q)\\
&= (1+\vp^2) \nabla \mu \cdot \nabla \Big( \frac{\vp}{1+\vp^2} \Big) + \frac{\vp}{1+\vp^2} \big( \partial_t \vp -S\big)\\
&\quad - \div\Big( \frac{\vp^2}{1+\vp^2} \nabla q\Big)+ \nabla \Big( \frac{\vp}{1+\vp^2} \Big)\cdot \big( \vp \nabla q \big).
\end{align*}
It is worth mentioning that the advantage of the above relation is that the latter right-hand side does not depend on higher derivatives of the chemical potential $\mu$ other than the first one. Recalling that $-\Delta q= S+\div(\vp \nabla \mu)$, we then find
\begin{align*}
-\Delta q &= (1+\vp^2) \nabla \mu\cdot \nabla \Big( \frac{\vp}{1+\vp^2}\Big)- \div \Big( \frac{\vp^2}{1+\vp^2} \nabla q\Big) \\
&\quad + \nabla \Big( \frac{\vp}{1+\vp^2}\Big) \cdot \big( \vp \nabla q \big)+ \frac{\vp}{1+\vp^2}(\partial_t \vp-S) + S, 
\end{align*} 
which is equivalent to 
\begin{align}
\label{comp}
-\div \Big( \frac{1}{1+\vp^2} \nabla q\Big)&= (1+\vp^2) \nabla \mu\cdot \nabla \Big( \frac{\vp}{1+\vp^2}\Big)+ \nabla \Big( \frac{\vp}{1+\vp^2}\Big) \cdot \big( \vp \nabla q \big) \notag\\
&\quad + \frac{\vp}{1+\vp^2}(\partial_t \vp-S) + S.
\end{align}
Multiplying \eqref{comp} with $(1+\vp^2)$ allows us to rewrite it further as follows
\begin{align}
\nonumber
-\Delta q&= (1+\vp^2)\nabla \Big(\frac{1}{1+\vp^2}\Big)\cdot \nabla q +(1+\vp^2)^2 \nabla \mu\cdot \nabla \Big( \frac{\vp}{1+\vp^2}\Big)\\
\nonumber
&\quad + (1+\vp^2) \vp \nabla \Big( \frac{\vp}{1+\vp^2}\Big) \cdot \nabla q+ \vp(\partial_t \vp-S) + (1+\vp^2)S\\
\nonumber
&= \Big( \frac{\vp(1-\vp^2)}{1+\vp^2}-\frac{2\vp}{1+\vp^2} \Big) \nabla \vp \cdot \nabla q
+ (1-\vp^2) \nabla \mu \cdot \nabla \vp+ \vp \partial_t \vp + S (1+\vp^2-\vp)\\
\label{comp2}
&=  \frac{-\vp-\vp^3}{1+\vp^2} \nabla \vp \cdot \nabla q
+ (1-\vp^2) \nabla \mu \cdot \nabla \vp+ \vp \partial_t \vp + S (1+\vp^2-\vp).
\end{align}
Here we have used 
$$
\nabla\Big( \frac{1}{1+\vp^2} \Big)= \frac{-2\vp}{(1+\vp^2)^2} \nabla \vp, \quad \nabla \Big( \frac{\vp}{1+\vp^2}\Big)= \frac{1-\vp^2}{(1+\vp^2)^2} \nabla \vp.
$$
Therefore, we summarize the above computations by reporting a second elliptic problem for the pressure
\begin{equation}
\label{PP2}
\begin{cases}
-\Delta q= \frac{-\vp-\vp^3}{1+\vp^2} \nabla \vp \cdot \nabla q
+ (1-\vp^2) \nabla \mu \cdot \nabla \vp+ \vp \partial_t \vp + S (1+\vp^2-\vp), \quad &\text{in } \ \Omega,\\
q=0,\quad &\text{on }\ \partial \Omega.
\end{cases}
\end{equation}
We now consider the elliptic problem for the chemical potential $\mu$. First, we have from \eqref{CHHS}$_3$ that
$$
-\Delta \mu = S-\partial_t \vp- \div (\uu \vp).
$$
Exploiting \eqref{CHHS}$_1$ and the boundary conditions \eqref{bc}, we observe that
\begin{align*}
\nabla \mu\cdot \n &= (\uu\cdot \n) \vp= - (\nabla q \cdot \n)\vp - (\vp \nabla \mu \cdot \n) \vp,
\end{align*} 
which can be rewritten as 
$$
(1+\vp^2)(\nabla \mu \cdot \n)= -(\nabla q \cdot \n)\vp.
$$
Then, using the boundary condition satisfied by $\vp$, we obtain 
$$
\partial_\n \mu= - \nabla \Big( \frac{\vp}{1+\vp^2} q\Big)\cdot \n.
$$
Summing up, we deduce the Neumann problem
\begin{equation}
\label{muP}
\begin{cases}
-\Delta \mu = S-\partial_t \vp- \div (\uu \vp), \quad &\text{in } \ \Omega,\\
\partial_\n \mu= - \nabla \Big( \frac{\vp}{1+\vp^2} q\Big)\cdot \n, \quad &\text{on }\ \partial \Omega.
\end{cases}
\end{equation}
Finally, in order to control the boundary term in \eqref{muP}, we also deduce an elliptic problem for the function $\frac{\vp}{1+\vp^2}q$.
By the relation \eqref{comp}, we have 
\begin{align}
\nonumber
& \Delta \Big( \frac{\vp}{1+\vp^2}q \Big) =
\div \Big( \frac{\vp}{1+\vp^2} \nabla q\Big)+ 
\div \Big( \nabla \big(\frac{\vp}{1+\vp^2}\big) q \Big)\\
\nonumber
& \quad= \vp \, 
\div \Big( \frac{\nabla q}{1+\vp^2}\Big)+ \frac{1}{1+\vp^2} \nabla q \cdot \nabla \vp+ \Delta \Big(\frac{\vp}{1+\vp^2}
\Big) q+ \nabla \Big( \frac{\vp}{1+\vp^2}\Big) \cdot \nabla q\\
\nonumber
& \quad = -(1+\vp^2) \vp \nabla \mu \cdot \nabla \Big( \frac{\vp}{1+\vp^2}\Big) - \vp^2 \nabla \Big( \frac{\vp}{1+\vp^2}\Big) \cdot \nabla q\\
\nonumber
&\quad\quad -\frac{\vp^2}{1+\vp^2} (\partial_t \vp- S)- \vp S+ \frac{1}{1+\vp^2} \nabla q \cdot \nabla \vp+ \Delta \Big(\frac{\vp}{1+\vp^2}
\Big) q+ \nabla \Big( \frac{\vp}{1+\vp^2}\Big) \cdot \nabla q\\
\nonumber
&\quad =-(1+\vp^2)\vp \frac{1-\vp^2}{(1+\vp^2)^2} \nabla \vp \cdot \nabla \mu + \Big( -\frac{1-\vp^2}{(1+\vp^2)^2} \vp^2+ \frac{1}{1+\vp^2} + \frac{1-\vp^2}{(1+\vp^2)^2} \Big) \nabla \vp\cdot \nabla q\\
\nonumber
&\quad\quad -\frac{\vp^2}{1+\vp^2}(\partial_t \vp-S)- \vp S+
\Delta \Big( \frac{\vp}{1+\vp^2} \Big) q\\
\nonumber
&\quad=-(1+\vp^2)\vp \frac{1-\vp^2}{(1+\vp^2)^2} \nabla \vp \cdot \nabla \mu + \Big( \frac{1+\vp^2+(1-\vp^2)^2}{(1+\vp^2)^2} \Big) \nabla \vp\cdot \nabla q\\
\label{ell:q}
&\quad\quad -\frac{\vp^2}{1+\vp^2}(\partial_t \vp-S)- \vp S+
\Delta \Big( \frac{\vp}{1+\vp^2} \Big) q.
\end{align} 
Thus, in light of the homogeneous Dirichlet boundary condition for $q$, we infer that
\begin{equation}
\label{phiqP}
\begin{cases}
\Delta \Big( \frac{\vp}{1+\vp^2}q \Big)= -\vp \frac{1-\vp^2}{1+\vp^2} \nabla \vp \cdot \nabla \mu + \Big( \frac{1+\vp^2+(1-\vp^2)^2}{(1+\vp^2)^2} \Big) \nabla \vp\cdot \nabla q\\
\qquad \qquad \quad \ \ \ -\frac{\vp^2}{1+\vp^2}(\partial_t \vp-S)- \vp S+\Delta \Big( \frac{\vp}{1+\vp^2} \Big) q, \quad &\text{in } \ \Omega,\\
 \frac{\vp}{1+\vp^2}q= 0, \quad &\text{on }\ \partial \Omega.
\end{cases}
\end{equation}

\section{Global Well-Posedness of Strong Solutions in Two Dimensions}
\label{S3}
\setcounter{equation}{0}

In this section we prove global existence and uniqueness of strong solutions to system \eqref{CHHS}-\eqref{bc} in two dimensions. 
Under the assumptions stated in Section \ref{pre}, our result is formulated as follows: 
\begin{theorem}\label{teo:2d}
Let $\Omega$ be a bounded domain with smooth boundary in $\mathbb{R}^2$. Assume the
above conditions\/ {\rm (A1)--(A3)} hold. Then, for any $T>0$, there exists a unique strong 
solution $(\uu,q,\vp)$ to system \eqref{CHHS}-\eqref{bc} such that
\begin{align}
\label{regou}
& \uu \in L^\infty(0,T; L^2(\Omega))\cap L^4(0,T;H^1(\Omega)), \\
\label{regoq}
& q \in L^\infty(0,T;H_0^1(\Omega))\cap L^4(0,T;H^2(\Omega)),\\
\label{regofhi}
& \vp \in L^\infty(0,T;W^{2,p}(\Omega))\cap H^1(0,T;H^1(\Omega)), \quad \forall \, p\in [1,\infty),\\
\label{bound-prop}
&\vp \in L^{\infty}(\Omega \times (0,T)): \ |\vp(x,t)|<1 \text{ a.e. }(x,t) \in \Omega\times(0,T),\\
\label{regomu}
& \mu \in L^\infty(0,T;H^1(\Omega)) \cap L^4(0,T;H^2(\Omega)),\\
\label{regopsi}
&\Psi'(\fhi), \Psi''(\fhi) \in L^\infty(0,T;L^p(\Omega)), \quad \forall \, p\in [1,\infty).
\end{align}
The strong solution satisfies the system \eqref{CHHS}-\eqref{bc} 
almost everywhere in $\Omega \times (0,\infty)$. In addition, it fulfills the initial value $\vp(\cdot,0)=\vp_0 (\cdot)$.
\end{theorem}


\subsection{A Priori Estimates}
\label{subsec:en}

In this part we prove the {\it a priori} estimates for to system \eqref{CHHS}-\eqref{bc} which are needed 
to establish the existence of a strong solutions as stated in Theorem~\ref{teo:2d}. We first derive the basic bounds resulting
from the evolution of the total mass and of the physical energy, and then we carry out the higher order estimates that entail the regularity \eqref{regou}-\eqref{regopsi} for the solution.  
For clarity of presentation, these estimates are performed in a formal way without referring to any explicit 
regularization or approximation of the system. A rigorous regularization strategy compatible with the estimates below could 
be written by following the lines of the argument in \cite{FLRS2018} (cf.~also \cite{G2020,GGW2018}). In particular, it is 
mentioning that the physical bound $\vp \in L^\infty(\Omega \times(0,T))$ such that $\|\vp\|_{L^\infty(\Omega \times (0,T))}\leq 1$, 
which holds in the limit because of the occurrence of the logarithmic nonlinearity $\Psi$,  is not usually conserved in such approximation 
schemes. In view of this issue, we will first show for simplicity the total mass and the basic energy estimates by considering a solution
that satisfies the physical bound \eqref{bound-prop} (see Total Mass Dynamics and Energy Estimates - Part I), and then we will adapt our 
argument for a solution which is not essentially bounded (see Energy Estimates - Part II for more details). Lastly, we will address the
crucial part of the proof which consists of the global-in-time higher order estimates of the solutions. Once again, also in that part we will proceed without using any boundedness assumption on the solution.  
\medskip

\noindent
\textbf{Total Mass Dynamics}.
Integrating \eqref{CHHS}$_3$ over $\Omega$, using the boundary conditions \eqref{bc} and the form of $S$, we find the evolution equation
$$
\ddt \overline{\vp}+ m\overline{\vp} = \frac{1}{|\Omega|} \int_{\Omega} h(\vp) \,  \d x.
$$
Solving this linear differential equation, we obtain
$$
\overline{\vp}(t)= \overline{\vp}(0)\mathrm{e}^{-mt}+
\mathrm{e}^{-mt} \int_0^t  \mathrm{e}^{ms}\frac{1}{|\Omega|} \int_{\Omega} h(\vp(x,s)) \,  \d x \, \d s,
$$
whence we can easily deduce the estimates
\begin{align}
\label{mass-law}
& \overline{\vp}(t) \le \overline{\vp}(0)\mathrm{e}^{-mt}
  + \frac{\hsu}{|\Omega|m} \big(1-\mathrm{e}^{-mt}\big),\\
\label{mass-law2}
& \overline{\vp}(t) \ge \overline{\vp}(0)\mathrm{e}^{-mt}
  + \frac{\hgiu}{|\Omega|m} \big(1-\mathrm{e}^{-mt}\big).
\end{align}
Then, viewing the above right-hand sides as convex combinations
and exploiting the assumption \eqref{hp:h}, we can easily deduce that
there exist two constants $c_1$ and $c_2$ depending only on 
$\overline{\vp}(0)$, $\Omega$, $h$ and $m$ such that
%
%
\begin{equation}
\label{mass}
-1<c_1\leq \overline{\vp}(t)\leq c_2 <1, \quad \forall \, t\geq 0.
\end{equation}
Note that this property holds both in the two and the three dimensional setting with no variation in the proof.
\medskip

\noindent
\textbf{Energy Estimates - Part I}.
We deduce the basic energy bound deriving from the variational structure of the system. Due to the presence of a mass source, the resulting variational equality will 
contain forcing terms on the right-hand side accounting for mass inflow. As for the mass conservation property \eqref{mass}, 
this part will also hold both for $d=2$ and $d=3$. 
As said, for convenience of presentation, we assume here the {\it a priori}\/ bound 
$\vp \in L^\infty(\Omega \times(0,T))$ such that 
\begin{equation}
\label{phi-bound}
\|\vp\|_{L^\infty(\Omega \times (0,T))}\leq 1.
\end{equation}
The argument below will be adapted in the subsequent section to a possible approximation of the system without taking advantage of the bound \eqref{phi-bound}.
\smallskip

We recall the total free energy associated with system \eqref{CHHS} (cf. \eqref{FE})
\begin{equation}
\label{energy}
\E(\vp)= \int_\Omega  \frac12 |\nabla \vp|^2
+ \Psi(\vp) \, \d x.
\end{equation}
We multiply \eqref{CHHS}$_1$ by $\uu$, \eqref{CHHS}$_3$ by $\mu$ and \eqref{CHHS}$_4$ 
by $\partial_t \vp$. Integrating over $\Omega$, using the boundary conditions \eqref{bc} 
and adding up the resulting relations, we obtain 
\begin{equation}
\label{EI}
\ddt E(\vp)+ \| \nabla \mu\|^2 + \| \uu\|^2 = \int_{\Omega} S q + S \mu \, \d x.
\end{equation}
In order to control the right-hand side, we recall the inequality (see, e.g., \cite{MZ04})
$$
c \|F'(\vp)\|_{L^1(\Omega)}\leq \int_{\Omega} (\vp-\overline{\vp})F'(\vp) \, \d x+ \widetilde{C},
$$
where we have set $F(s)=\Psi(s)+ \frac{\theta_0}{2}s^2$ (the convex part of the potential). Here, the positive constants $c$ and $\widetilde{C}$ may depend on $c_1$ and $c_2$.
Multiplying \eqref{CHHS}$_4$ by $\vp-\overline{\vp}$ and integrating over $\Omega$, we find
$$
\|\nabla \vp\|^2+ \int_{\Omega} (\vp-\overline{\vp})F'(\vp) \, \d x
= \int_{\Omega} \mu (\vp-\overline{\vp})\, \d x+ \theta_0 \int_{\Omega}\vp (\vp-\overline{\vp}) \, \d x.
$$
By using Poincar\'{e}'s inequality, the boundedness of $\vp$ and the above inequality, we have 
$$
\|\nabla \vp\|^2 + \|F'(\vp)\|_{L^1(\Omega)}\leq C(1+  \| \nabla \mu\|),
$$
for some positive constant $C$\footnote{In what follows, the notation $C$ will stand for a general constant depending on the parameters of the system which may vary from line to line.}.
Since $|\overline{\mu}|= |\overline{\Psi'(\vp)}|$, we obtain
\begin{equation}
\label{mass-mu}
|\overline{\mu}|\leq C(1+ \| \nabla \mu\|).
\end{equation}
This entails that 
\begin{equation}
\label{Smu-0}
\Big| \int_{\Omega} S \mu \, \d x \Big| \leq C(1+\| \nabla \mu\|).
\end{equation}
By the theory for the Laplace equation with Dirichlet boundary condition applied to the pressure problem \eqref{Smu2-0}, we have 
$$
\| q\|_{H^1(\Omega)}\leq C \| S\| + \| \vp \nabla \mu\|.
$$
Thus, using once more the boundedness of $\fhi$ and $h(\fhi)$, we deduce that
\begin{equation}
\label{Sq-0}
\Big| \int_{\Omega} S q \, \d x \Big| \leq C(1 + \| \nabla \mu\|).
\end{equation}
Combining \eqref{EI}, \eqref{Smu-0} and \eqref{Sq-0} together, we eventually arrive at
$$
\ddt E(\vp)+ \frac12 \| \nabla \mu\|^2 + \| \uu\|^2\leq C.
$$
Integrating on the time interval $[0,T]$, we deduce that
\begin{equation}
\label{energybalance}
 \E(\vp(t))+ \int_0^t \int_{\Omega} |\nabla \mu|^2 + |\uu|^2 \, \d x \, \d \tau \leq \E(\vp_0)+CT, \quad \forall \, t \geq 0.
\end{equation}
Thus, recalling also \eqref{mass-mu}, we can infer that
$$
\vp \in L^\infty(0,T;H^1(\Omega)), \quad \mu\in L^2(0,T;H^1(\Omega)),
  \quad \uu\in L^2(0,T;L^2(\Omega)), \quad q\in L^2(0,T;H^1_0(\Omega)).
$$
By the argument exploited in \cite{GGW2018} (cf. Lemmas 7.3 and 7.4, see also \cite[Theorem 2.2]{G2020}), we find
\begin{equation}
\label{fhi-reg1}
\vp \in L^4(0,T;H^2(\Omega))\cap L^2(0,T;W^{2,p}(\Omega)),
\end{equation}
where $p=6$ if $d=3$ and any finite $p$ if $d=2$. In addition, since 
$$
\| \partial_t \fhi\|_{\ast}\leq \| \nabla \mu\|+ \| \uu\| \| \fhi\|_{L^\infty(\Omega)}+ C \| S\|,
$$
it follows that
$$
\partial_t \vp \in L^2(0,T; (H^1(\Omega))').
$$
\begin{remark}
In light of \eqref{mass-mu}, the above argument can be easily modified in order to obtain a dissipative estimate (see \cite[Theorem 2.2]{GGW2018})
$$
\ddt E(\vp)+ k \Big( E(\vp)+ \| \nabla \mu\|^2+ \|\uu \|^2 \Big)\leq C,
$$
for some positive constants $k$ and $C$ which may depend on $m$, $h$, $c_1$, $c_2$, $\theta$, $\theta_0$ and $\Omega$, but are independent of the initial datum. 
\end{remark}
\medskip

\noindent
\textbf{Energy Estimates - Part II}. In view to adapt the calculation performed in the previous part to a possible approximation of the system, 
we will only assume that $\Psi$ has the form $\Psi(s) = F(s) - \frac{\theta_0}{2}s^2$, where $F$ 
is convex and goes at infinity at least as a polynomial of sufficiently large degree.
More precisely, we ask that\footnote{Here the exponent $5$ is probably not optimal and could be lowered. On the other
hand, this choice permits us to simplify the subsequent computations.}
\begin{equation}
\label{superquad}
  \liminf_{|r|\to \infty} \frac{F'(r)\sign(r)}{|r|^5} = 3 \kappa > 0.
\end{equation}
The above property is to be intended in the following way: as we consider, for $n\in {\mathbb N}$, a family $F_n$ of smooth approximants of the original 
(logarithmic) $F$ (see, e.g., \cite{FG2012, G2020, GGW2018}), \eqref{superquad} holds for $F_n$ with $\kappa$ independent of $n$.  Notice that, for the 
logarithmic $F$, the above property \eqref{superquad} may be thought to be valid with $\kappa=+\infty$ provided that we consider $F$ as a function taking values into ${\mathbb R}\cup\{+\infty\}$. 
For this reason, we will write the following part by using the notation
$F$ (rather than $F_n$), but the only coercivity property we will assume
is \eqref{superquad}. In addition to that, we assume that, as far as an approximation is concerned, the function $S=S(\fhi)$ needs 
to be extended for $|\fhi|\ge 1$. To this aim, we consider a regular extension $\tilde{h}$ of $h$ which satisfies the following properties:
\begin{itemize}
\item[$\bullet$]
The function $\tilde{h}: \mathbb{R}\rightarrow [\underline{h}-\varepsilon,\overline{h}+\varepsilon]$ is of class $\mathcal{C}^2(\mathbb{R})$ 
with bounded derivatives. In addition, it satisfies $\tilde{h}(s)=h(s)$ for $s\in [-1,1]$ and 
\begin{equation}
\label{hp:h2}
-1< \frac{\underline{h}-\varepsilon}{|\Omega| m}, \quad \frac{\overline{h}+\varepsilon}{|\Omega|m}<1.
\end{equation}
\end{itemize}
The existence of such function $\tilde{h}$ fulfilling the above properties can be obtained by a standard mollification procedure. Then, we define 
\begin{equation}
\label{S-ext}
S: \mathbb{R} \rightarrow \mathbb{R}, \quad S(s)= -m s + \tilde{h}(s).
\end{equation}
\smallskip

First of all, repeating the above argument for the total mass dynamics, we obtain
$$
\overline{\vp}(t)= \overline{\vp}(0)\mathrm{e}^{-mt}+
\mathrm{e}^{-mt} \int_0^t  \mathrm{e}^{ms}\frac{1}{|\Omega|} \int_{\Omega} \tilde{h}(\vp(x,s)) \,  \d x \, \d s. 
$$
As consequence, we find
\begin{equation}
\label{mass-law-ap}
 \overline{\vp}(0)\mathrm{e}^{-mt}
  + \frac{\hgiu-\varepsilon }{|\Omega|m} \big(1-\mathrm{e}^{-mt}\big)
  \leq \overline{\vp}(t) \le \overline{\vp}(0)\mathrm{e}^{-mt}
  + \frac{\hsu +\varepsilon}{|\Omega|m} \big(1-\mathrm{e}^{-mt}\big).
\end{equation}
Thanks to \eqref{hp:h2}, there exist $\tilde{c}_1$ and $\tilde{c_2}$ such that
\begin{equation}
\label{mass2}
-1<\tilde{c}_1\leq \overline{\vp}(t)\leq \tilde{c}_2 <1, \quad \forall \, t\geq 0.
\end{equation}
We proceed by recalling the energy equation (cf. \eqref{EI})
\begin{equation}
\label{EI-2}
\ddt E(\vp)+ \| \nabla \mu\|^2 + \| \uu\|^2 = \int_{\Omega} S q + S \mu \, \d x.
\end{equation}
The rest of the proof concerns the estimate of the right hand side.
To this aim, testing \eqref{CHHS}$_4$ by $\vp-\overline{\vp}$ 
and integrating over $\Omega$, we find
\begin{equation}
\label{g:01}
\|\nabla \vp\|^2+ \int_{\Omega} (\vp-\overline{\vp})F'(\vp) \, \d x
= \int_{\Omega} \mu (\vp-\overline{\vp})\, \d x+ \theta_0 \int_{\Omega}\vp (\vp-\overline{\vp}) \, \d x.
\end{equation}
We notice that, by the generalized Poincar\'{e} inequality \eqref{poincare} and the Young inequality, there holds
\begin{equation}
\label{g:01b}
  \int_{\Omega} \mu (\vp-\overline{\vp})\, \d x+ \theta_0 \int_{\Omega}\vp (\vp-\overline{\vp}) \, \d x
   \le \delta \| \nabla \mu \|^2 + C_\delta \| \nabla \fhi \|^2,
\end{equation}
for $\delta > 0$ to be chosen later, and correspondingly $C_\delta > 0$. 
We observe that 
\begin{equation} 
\label{g:01c}
 \frac12 \int_{\Omega} (\vp-\overline{\vp})F'(\vp) \, \d x \ge \kappa_1 \|F'(\vp)\|_{L^1(\Omega)} - C,
\end{equation}
where the positive constants $\kappa_1$ and $C$ may depend on $\tilde{c}_1$ and $\tilde{c}_2$
(cf.~\eqref{mass2})%
\footnote{If $F=F_n$ constitutes an approximation of the convex part of the logarithmic potential, then it can be shown 
that $\kappa_1$ and $C$ are also independent of $n$ (see \cite[Eqns (3.35)-(3.37)]{FG2012}).}. On the other hand, \eqref{mass} and \eqref{superquad} also entail that 
$$
  \frac12\int_{\Omega} (\vp-\overline{\vp})F'(\vp) \, \d x
   \ge \kappa \| \vp \|_{L^6(\Omega)}^6 - C.
$$
As a consequence, it follows from \eqref{g:01} that
\begin{equation}
\label{g:02}
  \kappa_1 \|F'(\vp)\|_{L^1(\Omega)}
   + \kappa \| \vp \|_{L^6(\Omega)}^6
  \le \delta \| \nabla \mu \|^2 + C_\delta \| \nabla \fhi \|^2 + C.
\end{equation}  
We multiply \eqref{g:02} by $M>0$ to be chosen later and sum the resulting inequality to \eqref{EI} to obtain
\begin{align}
\nonumber
 & \ddt E(\vp) + (1 - M\delta) \| \nabla \mu\|^2 
  + M  \kappa_1 \|F'(\vp)\|_{L^1(\Omega)}
  + M\kappa \| \vp \|_{L^6(\Omega)}^6  + \| \uu\|^2\\
 \label{EI2}
  & \quad\quad
   \le \int_{\Omega} S q + S \mu \, \d x
   + M C_\delta \| \nabla \fhi \|^2 + M C.
\end{align}
Now, integrating \eqref{CHHS}$_4$ over $\Omega$, we have
\begin{equation}
\label{g:02b}
  \ov{\mu} = \ov{\Psi'(\fhi)} 
    = \ov{F'(\fhi)} - \theta_0 \ov{\fhi}.
\end{equation}  
We notice that
\begin{align}
\nonumber
  \int_\Omega S \mu  \, \d x
  & = \int_\Omega S(\fhi) ( \mu - \ov{\mu} ) \, \d x
   + \ov{\mu} \int_\Omega S(\fhi) \, \d x\\
\nonumber
  & = \int_\Omega \big( S(\fhi) - \ov{S(\fhi)} \big) ( \mu - \ov{\mu} ) \, \d x
   + \ov{\mu} \int_\Omega S(\fhi)\, \d x \\
\nonumber
  & \le C \| S'(\fhi) \nabla \fhi \| \| \nabla \mu \|
   + | \ov{\mu}| \Big| \int_\Omega -m \fhi +\tilde{h}(\fhi) \, \d x \Big|\\
\nonumber
  & \le C (m+\| \tilde{h}'(\fhi)\|_{L^\infty(\Omega)}) \| \nabla \fhi \| \| \nabla \mu \|
   +C |\Omega| | \ov{\mu}|\\
\label{g:0x}
  & \le \frac14 \| \nabla \mu \|^2 + C \| \nabla \fhi \|^2 
   + C | \Omega | | \ov{F'(\fhi)} | + C (1 + \| \fhi \|^2 ).
\end{align}
Here we used \eqref{S-ext}, \eqref{mass2}, \eqref{g:02b} and the fact that $|\tilde{h}'(s)|\leq C$ for all $s\in \mathbb{R}$.
%
%
%
%
%
 %
 %
 %
%
Next, we consider the problem for the pressure $q$ obtained 
from \eqref{CHHS}$_1$-\eqref{CHHS}$_2$
\begin{equation}
\label{Smu2}
\begin{cases}
-\Delta q= S+ \div (\vp \nabla \mu) \quad &\text{in } \ \Omega,\\
q=0\,\quad &\text{on }\ \partial \Omega.
\end{cases}
\end{equation}
By the regularity theory for the Laplace equation with Dirichlet boundary condition, we have
\begin{equation}
\label{Sq1}
\| q\|_{W^{1,\frac32}(\Omega)} \leq C \| S\| + \| \vp \nabla \mu\|_{L^\frac32(\Omega)}
 \leq C(1+ \| \vp\|) + \| \vp \|_{L^{6}(\Omega)} \| \nabla \mu\|.
\end{equation}
Hence, using also Sobolev's embeddings,
\begin{align}
\nonumber
 \Big| \int_{\Omega} S q \, \d x\Big|
  & \le \| S(\fhi) \| \| q \|  \le C(1+\| \vp\|) \big( 1 +\|\vp \|+ \| \vp \|_{L^{6}(\Omega)} \| \nabla \mu\| \big)\\
\nonumber
  & \le \frac14 \| \nabla \mu \|^2 + C(1+\| \fhi\|^2) + C \| \fhi \|_{L^{6}(\Omega)}^4\\
   \label{Sq2}
&  \le \frac14 \| \nabla \mu \|^2 + \frac\kappa2 \| \fhi \|_{L^{6}(\Omega)}^6 + C(1+\| \fhi\|^2).
\end{align}
Then, replacing \eqref{g:0x} and \eqref{Sq2} into \eqref{EI2}, we deduce
\begin{align}
\nonumber
 & \ddt E(\vp) + \Big(\frac12 - M\delta\Big) \| \nabla \mu\|^2 
  + M  \kappa_1 \|F'(\vp)\|_{L^1(\Omega)}
  + \Big( M\kappa - \frac\kappa2 \Big) \| \vp \|_{L^6(\Omega)}^6  + \| \uu\|^2\\
 \label{EI3}
  & \quad\quad
   \le C | \Omega | \bigg| \int_\Omega F'(\fhi) \bigg| + C (1 + \| \fhi \|^2 )
   + M C_\delta \| \nabla \fhi \|^2 + M C.
\end{align}
Now, taking first $M\ge \max\{2^{-1}+\kappa^{-1},(1+C|\Omega|)\kappa_1^{-1}\}$ and subsequently
$\delta \le (4M)^{-1}$, we readily deduce that
\begin{equation}
\label{E4}
 \ddt E(\vp) + \frac14 \| \nabla \mu\|^2 
  + \|F'(\vp)\|_{L^1(\Omega)}
  + \| \vp \|_{L^6(\Omega)}^6  + \| \uu\|^2
   \le C (1 + \| \fhi \|^2 )
   + C \| \nabla \fhi \|^2.
\end{equation}
Exploiting once again the generalized Poincar\'{e} inequality \eqref{poincare} and the total mass bound \eqref{mass2} to estimate the first term on the right hand side, we eventually find the differential inequality
\begin{equation}
\label{EI5}
 \ddt E(\vp) + \frac14 \| \nabla \mu\|^2 
  + \|F'(\vp)\|_{L^1(\Omega)}
  + \frac12\| \vp \|_{L^6(\Omega)}^6  + \| \uu\|^2
 \le  C \big( 1 + \| \nabla \fhi \|^2 \big).
\end{equation}
It is worth noting that the above relation is independent on 
any eventual approximation parameter provided that $S$ is extended as in \eqref{S-ext} and $F$ is approximated in such a way that \eqref{superquad} is satisfied uniformly with respect to the approximation parameters. An application of the Gronwall lemma entails that
\begin{equation}
\label{g:01d}
  \vp \in L^\infty(0,T;H^1(\Omega)), \quad \nabla\mu\in L^2(0,T;L^2(\Omega)),
  \quad \uu\in L^2(0,T;L^2(\Omega)).
\end{equation}
Let us now go back to \eqref{g:01}. Using \eqref{g:01c} and estimating the right hand side without using Young's inequality as was done in \eqref{g:01b}, we easily get
\begin{equation}
\label{g:01e}
  2\kappa_1 \| F'(\fhi) \|_{L^1(\Omega)} 
   \le C + C \| \nabla \fhi \|^2 + C \| \nabla \fhi \| \| \nabla \mu \|.
\end{equation}
Thus, using \eqref{g:01d} and integrating in time, we deduce that
$$
  F'(\fhi) \in L^2(0,T;L^1(\Omega)).
$$
Moreover, on account of  \eqref{g:02b}, we have from \eqref{g:01e}
\begin{equation}
\label{g:01f}
  \| \mu \|_{H^1(\Omega)} \le C ( 1 + \| \nabla \mu \| ),
\end{equation}
whence, recalling \eqref{g:01d},
$$
  \mu \in L^2(0,T;H^1(\Omega)).
$$
Also, owing to \eqref{Sq1}, it is not difficult to obtain
$$
  q \in L^2(0,T;W^{1,\frac32}(\Omega)).
$$
Since $F$ is convex, arguing as in \cite[Section 3.2]{GGW2018} and in \cite[Proof of Theorem 5.1, Step 5]{G2020}, we find
\begin{equation}
\label{g:15}
\vp \in L^4(0,T;H^2(\Omega))\cap L^2(0,T;W^{2,p}(\Omega)),
\end{equation}
where $p=6$ if $d=3$ and any finite $p$ if $d=2$. 
Finally, we infer from \eqref{CHHS}$_3$ and \eqref{Ad3} that
\begin{align*}
\| \partial_t \vp\|_{\ast}
&\leq \| \nabla \mu\|+ \| \uu\| \| \fhi\|_{L^\infty(\Omega)}+ C \| S\|\\
&\leq \| \nabla \mu\|+ \| \uu\| \| \fhi\|_{H^1(\Omega)}^\frac12 \|\fhi \|_{H^2(\Omega)}^\frac12+ C \| S\|.
\end{align*}
In light of \eqref{g:01d}, \eqref{g:01f} and \eqref{g:15}, it follows that (if $d=3$)
$$
\partial_t \vp \in L^\frac85(0,T; (H^1(\Omega))').
$$
On the other hand, in the two dimensional case $d=2$, by exploiting the Brezis--Gallouet--Wainger inequality (cf. \eqref{BW2} below),
it is possible to show that $\partial_t \vp \in L^q(0,T; (H^1(\Omega))')$ for any $q \in [1,2)$.
\medskip

\noindent
\textbf{Higher Order Estimates in Two Dimensions}.
In this section we show the crucial global-in-time higher order estimates of the solution to system \eqref{CHHS}-\eqref{bc}. In turn, these bounds will imply the existence of global strong solutions in two dimensions.
As previously mentioned, we assume that the logarithmic potential is approximated by a sequence of polynomial functions as in the previous part Energy Estimates - Part II. In particular, we will not make use of the bound
 $\| \vp\|_{L^\infty(\Omega \times (0,T))}\leq 1$ since is not at our disposal in the approximation scheme\footnote{We also point out that we do not make use either of the regularity $\vp \in L^2(0,T;H^3(\Omega))$, 
which is available for weak solutions to \eqref{CHHS} with polynomial  potential (cf. \cite[Section 4.2]{G2020}). However, this regularity property does not uniformly hold in the approximation procedure.}.
\smallskip

\noindent
Let us start with the {\it a priori} bounds resulting from the energy balance \eqref{energybalance}
\begin{equation}
\label{EE}
\|\vp\|_{L^\infty(0,T;H^1(\Omega))}\leq C, \quad \|\mu\|_{L^2(0,T;H^1(\Omega))}\leq C,\quad \|\uu\|_{L^2(0,T;L^2(\Omega))}\leq C.
\end{equation}
As an immediate consequence, since $S= -m \vp + \tilde{h}(\vp)$, it is easily seen from \eqref{EE} that 
\begin{equation}
\label{S-H1}
\|S \|_{L^\infty(0,T;H^1(\Omega))}\leq C.
\end{equation}
Recalling the estimate (cf. \eqref{g:01f})
\begin{equation}
\label{muH1}
\| \mu\|_{H^1(\Omega)}\leq C(1+\| \nabla \mu\|),
\end{equation}
we have from \cite[Theorem 5.1]{G2020} (cf. \cite[Lemmas 7.3 and 7.4]{GGW2018}) that
\begin{equation}
\label{H2}
\| \vp\|_{H^2(\Omega)}^2 \leq C(1+ \| \nabla \mu\|)
\end{equation}
and 
\begin{equation}
\label{W2p}
\| \vp\|_{W^{2,p}(\Omega)}\leq C(1+\| \nabla \mu\|),
\end{equation}
for any $2\leq p<\infty$. 
Next, in order to control
 the $L^\infty$-norm of $\vp$, we will make use of the Brezis--Gallouet--Wainger inequality \eqref{BWd2} in the following form
\begin{equation}
\label{BW2}
\| f\|_{L^\infty(\Omega)}\leq C \| f\|_{H^1(\Omega)} \log^\frac12 \big( e+ \| f\|_{W^{1,r}(\Omega)}\big)+ C, 
\end{equation}
where $r>2$ and $e = \exp(1)$.  We are now in position to estimate $\partial_t \vp$. By using the boundary conditions \eqref{bc}, the estimates \eqref{S-H1} and \eqref{BW2}, we obtain
\begin{align*}
\| \partial_t \vp\|_{*}&\leq \| \nabla \mu\|+ \| \uu \vp\| +  C \|S\| \\
&\leq \| \nabla \mu\|+ \| \uu\| \|\vp\|_{L^\infty(\Omega)} + C\\
&\leq \| \nabla \mu\|+ C\| \uu\| \| \vp\|_{H^1(\Omega)} \log^\frac12 \big( e+ \| \vp\|_{W^{1,q}(\Omega)}\big)+ C\| \uu\|+ C, 
\end{align*}
for some $q>2$.
By using \eqref{EE} and \eqref{W2p}, we deduce that
\begin{equation}
\label{vpt}
\| \partial_t \vp\|_{*} \leq C (1+\| \nabla \mu\|+ \| \uu\|)\log^{\frac12} \big(e+\| \nabla \mu\|\big).
\end{equation}
Similarly, for the pressure we have
\begin{align*}
\| \nabla q \|
&\leq C \| S\| + \| \vp \nabla \mu\|\\
&\leq C+ \| \vp\|_{L^\infty(\Omega)} \|\nabla \mu\| \\
&\leq C+ C\| \nabla \mu\| \| \vp\|_{H^1(\Omega)} \log^\frac12 \big( e+ \| \vp\|_{W^{1,q}(\Omega)} \big),
\end{align*}
for some $q>2$.
Then, we find
\begin{equation}
\label{qH1}
\| \nabla q\|
\leq C(1+ \| \nabla \mu\|) \log^\frac12 \big( e+ \| \nabla \mu\| \big).
\end{equation}
In addition, we have
\begin{align*}
\| q\|&\leq  \|(-\Delta)^{-1} S\|+ \| (-\Delta)^{-1} \div(\vp \nabla \mu)\|\\
&\leq C \| S\| +C \| \vp \nabla \mu\|_{L^{\frac{6}{5}}(\Omega)}\\
&\leq C \| S\|+C \| \vp\|_{L^3(\Omega)} \| \nabla \mu\|,
\end{align*}
where $-\Delta$ denotes the Laplacian operator with Dirichlet boundary condition. 
By using \eqref{EE} and \eqref{S-H1}, we get
\begin{equation}
\label{qL2}
\| q\| \leq C(1+  \| \nabla \mu\|).
\end{equation}
Also, it easily follows from \eqref{S-H1}, \eqref{muH1} and \eqref{qL2} that
\begin{equation}
\label{Smu}
\Big| \int_{\Omega} S \mu \, \d x \Big| \leq C(1+\| \nabla \mu\|), 
\end{equation}
and
\begin{equation}
\label{Sq}
\Big| \int_{\Omega} S q \, \d x \Big| \leq C(1+  \| \nabla \mu\|).
\end{equation}

We now differentiate in time \eqref{CHHS}$_1$, multiply it by $\uu$ and integrate over $\Omega$. We obtain
\begin{equation}
\label{2nd}
\frac12 \ddt \| \uu\|^2 + (\nabla \partial_t q, \uu)= -\int_{\Omega} \partial_t \vp \nabla \mu \cdot \uu \, \d x - \int_{\Omega} \vp \nabla \partial_t \mu \cdot \uu \, \d x.
\end{equation}
We observe that the second term on the left-hand side can be rewritten as 
$$
(\nabla \partial_t q, \uu)= -(\partial_t q, \div \uu)=-(\partial_t q, S)=
- \ddt \Big[(q,S)\Big] + (q,\partial_t S).
$$
Here we have used that $q=0$ on $\partial \Omega$.
Combining the two equations above, we have 
\begin{equation}
\label{test1}
\ddt \bigg[ \frac12 \| \uu\|^2 -(q,S) \bigg] 
= -\int_{\Omega} \partial_t \vp \nabla \mu \cdot \uu \, \d x - \int_{\Omega} \vp \nabla \partial_t \mu \cdot \uu \, \d x- (q, \partial_t S).
\end{equation}
Next, multiplying \eqref{CHHS}$_3$ by $\partial_t \mu$, we obtain
$$
(\partial_t \vp, \partial_t \mu)- (\Delta \mu, \partial_t \mu)= -(\div(\uu \vp), \partial_t \mu ) +(S, \partial_t \mu). 
$$
Exploiting the boundary conditions \eqref{bc}, we then have
\begin{equation}
\label{g:11}
(\partial_t \vp, \partial_t \mu) + \frac12 \ddt \| \nabla \mu\|^2 
= \int_{\Omega}\vp \uu \cdot \nabla \partial_t \mu \, \d x + (S, \partial_t \mu). 
\end{equation}
By using the expression \eqref{CHHS}$_4$ for 
$\mu$ and the boundary conditions \eqref{bc}, we get
$$
(\partial_t \vp, \partial_t \mu) = 
\|\nabla \partial_t \vp\|^2 + \int_{\Omega} F''(\vp) |\partial_t \vp |^2 \, \d x- \theta_0 \int_{\Omega} |\partial_t \vp|^2 \, \d x.
$$
Moreover, we rewrite the last term on the right-hand side of \eqref{g:11} as 
$$
(S, \partial_t \mu)= \ddt\Big[(S,\mu)\Big]-
(\partial_t S, \mu).
$$
Then, combining the above relations, we find
\begin{align}
\ddt \bigg[ \frac12 \| \nabla \mu\|^2 &- (S,\mu)\bigg] + \|\nabla \partial_t \vp\|^2 + \int_{\Omega} F''(\vp) |\partial_t \vp |^2 \, \d x \notag \\
&= \theta_0 \int_{\Omega} |\partial_t \vp|^2 \, \d x- (\partial_t S, \mu) +\int_{\Omega}\vp \uu \cdot \nabla \partial_t \mu \, \d x. 
\label{test2}
\end{align}
Adding \eqref{test1} and \eqref{test2}, we obtain
\begin{align}
\ddt \bigg[ \frac12 \| \nabla \mu\|^2 &+ \frac12 \| \uu\|^2 - (S,\mu)- (S,q) \bigg] +
 \|\nabla \partial_t \vp\|^2 + \int_{\Omega} F''(\vp) |\partial_t \vp |^2 \, \d x \notag \\
&=  \theta_0 \int_{\Omega} |\partial_t \vp|^2 \, \d x- (\partial_t S, \mu)- (\partial_t S, q)-
\int_{\Omega} \partial_t \vp \nabla \mu \cdot \uu \, \d x
\notag \\
&=: I_1+I_2+I_3+I_4.
\label{test3}
\end{align}
Before proceeding to estimate the right-hand side in \eqref{test3}, we define 
\begin{equation}
\label{defiH}
H=  \frac12 \| \nabla \mu\|^2 + \frac12 \| \uu\|^2 - (S,\mu)- (S,q),
\end{equation}
and we observe that, in light of  \eqref{EE}, \eqref{Smu}, \eqref{Sq}, we have
\begin{equation}
\label{H}
\frac14 \| \nabla \mu\|^2 + \frac14 \| \uu\|^2-C \leq H \leq C (1+ \| \nabla \mu\|^2+ \| \uu\|^2 ).
\end{equation}
Now, since
\begin{equation}
\label{mean-est}
| \overline{v} |= \Big| \frac{1}{|\Omega|} \l v,1\r \Big| 
 \leq C \|v \|_{\ast}, 
\end{equation}
by the generalized Poincar\'{e}'s inequality \eqref{poincare}, we notice that
$$
  \| \partial_t \vp\|_{H^1(\Omega)} 
   \leq C ( \| \nabla \partial_t \vp\|+ |\overline{\partial_t \vp}|) 
   \leq C ( \| \nabla \partial_t \vp\|+ \|\partial_t \vp \|_{\ast}).
$$
Thus, we rewrite \eqref{test3} as follows
\begin{equation}
\label{test4}
\ddt H + \eta \|\partial_t \vp \|_{H^1(\Omega)}^2 \leq C \|\partial_t \vp \|_{\ast}^2 + I_1+I_2+I_3+I_4,
\end{equation}
for some positive $\eta$.
By duality, we have 
\begin{equation*}
I_1=\theta_0 \int_{\Omega} |\partial_t \vp|^2 \, \d x\leq \theta_0 \| \partial_t \vp\|_{H^1(\Omega)} \| \partial_t \vp\|_{\ast} \leq \frac{\eta}{8} \|  \partial_t \vp\|_{H^1(\Omega)}^2+ C \|\partial_t \vp \|_{\ast}^2.
\end{equation*}
Hence, by \eqref{vpt} and \eqref{H}, we find
\begin{equation}
\label{I1}
I_1 \leq \frac{\eta}{8} \|  \partial_t \vp\|_{H^1(\Omega)}^2+ C(C+ H)\log(C+ H).
\end{equation}
By definition, $\partial_t S= -m \partial_t \vp+ \tilde{h}'(\vp)\partial_t \vp$. Owing to \eqref{qL2}, this entails 
\begin{align*}
I_2+I_3&\leq C \| \partial_t \vp\| (\| \mu\|+ \| q\|) \\
&\leq 
C \| \partial_t \vp\|_{H^1(\Omega)}^\frac12 \| \partial_t \vp\|_{\ast}^\frac12 (1+\| \nabla \mu\|)\\
&\leq \frac{\eta}{4}  \|\partial_t \vp \|_{H^1(\Omega)}^2 + C \| \partial_t \vp\|_{\ast}^\frac23 (1+\| \nabla \mu\|)^\frac43.
\end{align*}
In light of \eqref{vpt} and \eqref{H}, we obtain
\begin{equation}
\label{I2-3}
I_2+I_3 \leq \frac{\eta}{4} \|\partial_t \vp \|_{H^1(\Omega)}^2 + C(C+H)\log^\frac13(C+H).
\end{equation}
The main task is now to estimate $I_4$. To do so, we exploit the elliptic structure of the system \eqref{CHHS}-\eqref{bc} to derive $H^2$ estimates for the pressure and the chemical potential. 
First, the regularity theory of the Laplace problem with Dirichlet boundary condition applied to \eqref{PP2} entails that
\begin{align*}
\| q\|_{H^2(\Omega)} \leq C \Big\| \vp \nabla \vp \cdot \nabla q\Big\|
+ C \|(1-\vp^2) \nabla \mu \cdot \nabla \vp\|
+ C \| \vp \partial_t \vp \|+ C \|S (1+\vp^2-\vp)\|.
\end{align*}
By using \eqref{EE} and \eqref{S-H1}, we find
\begin{align}
\| q\|_{H^2(\Omega)}
&\leq C \| \vp\|_{L^\infty(\Omega)} \| \nabla \vp\|_{L^\infty(\Omega)} \|\nabla q \|+ C (1+\| \vp\|_{L^\infty(\Omega)}^2)\| \nabla \vp\|_{L^\infty(\Omega)} 
\| \nabla \mu\|\notag \\
&\quad + C \| \vp\|_{L^\infty(\Omega)} \|\partial_t \vp \| + C.
\label{qH2}
\end{align}
Now, recalling the problem \eqref{muP}, using first the regularity theory of the Laplace equation with Neumann boundary condition
and the trace theorem for normal derivatives, we obtain 
\begin{align*}
 \| \mu\|_{H^2(\Omega)}
 & \leq C|\overline{\mu}|+ C \| \Delta \mu\|
   + C \Big\| - \nabla \Big( \frac{\vp}{1+\vp^2} q\Big)\cdot \n \Big\|_{H^{\frac12}(\partial \Omega)}\\
 & \leq C|\overline{\mu}| + C \| \Delta \mu\|
   + C \Big\| \frac{\vp}{1+\vp^2} q \Big\|_{H^2(\Omega)}. 
\end{align*}
By exploiting the equation \eqref{CHHS}$_2$ and the estimates \eqref{EE}, \eqref{S-H1} and \eqref{muH1}, we deduce that
\begin{align}\nonumber
  \| \mu\|_{H^2(\Omega)}&
    \leq  C(1+\| \nabla \mu\|)+ C \| S\|+ C\| \partial_t \vp\|+C \| \div(\uu \vp)\|
     +C \Big\| \frac{\vp}{1+\vp^2} q\Big\|_{H^2(\Omega)}\\
 \label{g:12}
  & \leq C(1+\| \nabla \mu\|)+ C\| \partial_t \vp\|+ C \| \uu \cdot \nabla \vp\|
   + C \Big\| \frac{\vp}{1+\vp^2} q\Big\|_{H^2(\Omega)}.
\end{align}
In order to provide a control of the last term on the right-hand side, we recall the elliptic problem \eqref{phiqP}.
By the regularity theory of the Laplace problem with Dirichlet boundary condition, we get
\begin{align}
\notag
\Big\|\frac{\vp}{1+\vp^2} q\Big\|_{H^2(\Omega)}
&\leq C\Big\| \Delta \Big( \frac{\vp}{1+\vp^2} q\Big)\Big\|\\
\notag
&\leq C \Big\|\vp \frac{1-\vp^2}{1+\vp^2} \nabla \vp \cdot \nabla \mu \Big\| + C \Big\| \Big( \frac{1+\vp^2+(1-\vp^2)^2}{(1+\vp^2)^2} \Big) \nabla \vp\cdot \nabla q\Big\|\\
&\quad + C \Big\| \frac{\vp^2}{1+\vp^2}(\partial_t \vp-S)\Big\|+ C \|\vp S\| + C \Big\| \frac{\vp}{1+\vp^2}\Big\|_{H^2(\Omega)} \| q\|_{L^\infty(\Omega)}.
\label{ell:q2}
\end{align}
We notice that 
$$
 \Big\| \frac{1-\vp^2}{1+\vp^2} \Big\|_{L^\infty(\Omega)}\leq C, \quad 
\Big\|  \frac{1+\vp^2+(1-\vp^2)^2}{(1+\vp^2)^2} \Big\|_{L^\infty(\Omega)}\leq C.
$$
In addition, let us define the function $g:\mathbb{R}\rightarrow \mathbb{R}$ as $g(x)= \frac{x}{1+x^2}$. It is clear that $|g(x)|\leq 1$ and its derivative $g'(x)=\frac{1-x^2}{(1+x^2)^2}$ and $g''(x)= \frac{2x^3-6x}{(1+x^2)^3}$ 
are such that $|g'(x)|\leq 1$ and $|g''(x)|\leq 2$. Then, by using \eqref{L} and \eqref{EE}, we find
\begin{align}
\Big\| \frac{\vp}{1+\vp^2}\Big\|_{H^2(\Omega)} &\leq 
C \Big\| \frac{\vp}{1+\vp^2}\Big\|_{L^2(\Omega)}
+ C \Big( \sum_{i,j=1}^2 \int_{\Omega} \big( \partial_{x_i} \partial_{x_j} \frac{\vp}{1+\vp^2} \big)^2 \, \d x \Big)^\frac12 \notag \\
&\leq C + C \Big( \sum_{i,j=1}^2 \int_{\Omega} \big(g'(\vp) \partial_{x_i} \partial_{x_j} \vp+ g''(\vp) \partial_{x_i}\vp \partial_{x_j}\vp\big)^2  \, \d x \Big)^\frac12 \notag \\
&\leq C \big( 1+ \| \nabla \vp\|_{L^4(\Omega)}^2+ \|\vp\|_{H^2(\Omega)} \big) \notag \\
&\leq C\big(1+ \| \vp\|_{H^2(\Omega)}\big).
\label{gfhi}
\end{align}
Thus, combining the above estimates, we arrive at
\begin{align}
\nonumber
\Big\|\frac{\vp}{1+\vp^2} q\Big\|_{H^2(\Omega)}
&\leq C \| \vp\|_{L^\infty(\Omega)} \| \nabla \vp\|_{L^\infty(\Omega)} \|\nabla \mu \|
+ C \| \nabla \vp\|_{L^\infty(\Omega)} \| \nabla q\|+
C \| \partial_t \vp\| \\
&\quad + C(1+\|\vp \|_{H^2(\Omega)}) \| q\|_{L^\infty(\Omega)}+C.
\label{phiqH2}
\end{align}
Then, going back to \eqref{g:12}, it is easy to deduce that
\begin{align}
\nonumber
  \| \mu\|_{H^2(\Omega)}
  & \leq C(1+\| \nabla \mu\|) + C\| \partial_t \vp\|
   + C \| \uu \| \|\nabla \vp \|_{L^\infty(\Omega)}
   + C(1+\|\vp \|_{H^2(\Omega)}) \| q\|_{L^\infty(\Omega)}\\
\label{muH2}
 & \quad 
   + C \| \vp\|_{L^\infty(\Omega)} \| \nabla \vp\|_{L^\infty(\Omega)} \|\nabla \mu \| 
   + C \| \nabla \vp\|_{L^\infty(\Omega)} \| \nabla q\|.
\end{align}

We are now in position to estimate $I_4$ in \eqref{test3}. By \eqref{L} 
and \eqref{muH2}, we have
\begin{align*}
I_4&= -\int_{\Omega} \partial_t \vp \nabla \mu \cdot \uu \, \d x
\leq \| \partial_t \vp \|_{L^4(\Omega)} \| \nabla \mu\|_{L^4(\Omega)} \| \uu\|\\
&\leq C \|\partial_t \vp \|^\frac12 \| \partial_t \vp\|_{H^1(\Omega)}^\frac12
\| \nabla \mu\|^\frac12 \| \mu\|_{H^2(\Omega)}^\frac12 \| \uu\|\\
&\leq C \| \partial_t \vp\|_{\ast}^\frac14 
\| \partial_t \vp\|_{H^1(\Omega)}^\frac34  
\| \nabla \mu\|^\frac12 \| \uu\| \\
& \quad \times 
\big( 1+ \| \nabla \mu\|^\frac12 + \| \partial_t \vp\|^\frac12+
\| \uu\|^\frac12  \|\nabla \vp\|_{L^\infty(\Omega)}^\frac12+
(1+\| \vp\|_{H^2(\Omega)}^\frac12) \| q\|_{L^\infty(\Omega)}^\frac12\\
& \quad\quad +\| \vp\|_{L^\infty(\Omega)}^\frac12 \| \nabla \vp\|_{L^\infty(\Omega)}^\frac12 \|\nabla \mu \|^\frac12
 + \| \nabla \vp\|_{L^\infty(\Omega)}^\frac12 \| \nabla q\|^\frac12 \big)\\
&=:  I_{41}+I_{42}+I_{43}+I_{44}+I_{45} +I_{46} +I_{47}.
\end{align*}
By using \eqref{H2}, \eqref{vpt}, \eqref{H} and Young's inequality, we have
\begin{align*}
I_{41}&= C \| \partial_t \vp\|_{\ast}^\frac14 
\| \partial_t \vp\|_{H^1(\Omega)}^\frac34  
\| \nabla \mu\|^\frac12 \| \uu\|\\
&\leq \frac{\eta}{56}\| \partial_t \vp\|_{H^1(\Omega)}^2
+C \|\partial_t \vp \|_{\ast}^\frac25 \|\nabla \mu \|^\frac45 \| \uu\|^\frac85\\
&\leq \frac{\eta}{56}\| \partial_t \vp\|_{H^1(\Omega)}^2
+ C (C+ H)^\frac15 \log^\frac15 (C+ H)(C+H)^\frac25(C+H)^\frac45\\
&\leq \frac{\eta}{56}\| \partial_t \vp\|_{H^1(\Omega)}^2+
C (C+ H)^\frac75 \log^\frac15 (C+ H),
\end{align*}
\begin{align*}
I_{42}&= C \| \partial_t \vp\|_{\ast}^\frac14 
\| \partial_t \vp\|_{H^1(\Omega)}^\frac34  
\| \nabla \mu\| \| \uu\|\\
&\leq \frac{\eta}{56}\| \partial_t \vp\|_{H^1(\Omega)}^2
+C \|\partial_t \vp \|_{\ast}^\frac25 \|\nabla \mu \|^\frac85 \| \uu\|^\frac85\\
&\leq \frac{\eta}{56}\| \partial_t \vp\|_{H^1(\Omega)}^2
+ C (C+ H)^\frac15 \log^\frac15 (C+ H)(C+H)^\frac45(C+H)^\frac45\\
&\leq \frac{\eta}{56}\| \partial_t \vp\|_{H^1(\Omega)}^2+
C (C+ H)^\frac95 \log^\frac15 (C+ H),
\end{align*}
and
\begin{align*}
I_{43}&= C \| \partial_t \vp\|_{\ast}^\frac14 
\| \partial_t \vp\|_{H^1(\Omega)}^\frac34  
\| \nabla \mu\|^\frac12 \| \uu\| \| \partial_t \vp\|^\frac12\\
&\leq C \| \partial_t \vp\|_{\ast}^\frac12 
\| \partial_t \vp\|_{H^1(\Omega)}  
\| \nabla \mu\|^\frac12 \| \uu\|\\
&\leq \frac{\eta}{56} \|\partial_t \vp \|_{H^1(\Omega)}^2
+C \| \partial_t \vp\|_{\ast} 
\| \nabla \mu\| \| \uu\|^2\\
&\leq \frac{\eta}{56} \|\partial_t \vp \|_{H^1(\Omega)}^2
+C(C+H)^\frac12 \log^\frac12(C+H) (C+H)^\frac12 (C+H)\\
&\leq \frac{\eta}{56} \|\partial_t \vp \|_{H^1(\Omega)}^2
+C(1+H)^2 \log^\frac12(C+H).
\end{align*}
Exploiting \eqref{H2}, \eqref{W2p}, \eqref{BW2}, \eqref{vpt}, \eqref{H} 
and Young's inequality again, we obtain
\begin{align*}
I_{44}&= C \| \partial_t \vp\|_{\ast}^\frac14 
\| \partial_t \vp\|_{H^1(\Omega)}^\frac34  
\| \nabla \mu\|^\frac12 \| \uu\|^\frac32  \| \nabla \vp\|_{L^\infty(\Omega)}^\frac12\\
&\leq \frac{\eta}{56} \| \partial_t \vp\|_{H^1(\Omega)}^2
+C \| \partial_t \vp\|_{\ast}^\frac25 \| \nabla \mu\|^\frac45 \|\uu \|^\frac{12}{5} \| \nabla \vp\|_{L^\infty(\Omega)}^\frac45\\
&\leq \frac{\eta}{56} \| \partial_t \vp\|_{H^1(\Omega)}^2
+C (C+H)^\frac15 \log^\frac15 (C+H) (C+H)^\frac85 
\Big( C\| \vp\|_{H^2(\Omega)} \log^\frac12(e+\| \vp\|_{W^{2,3}(\Omega)})+C\Big)^\frac45\\
&\leq \frac{\eta}{56} \| \partial_t \vp\|_{H^1(\Omega)}^2
+C (C+H)^\frac95 \log^\frac15(C+H) 
\Big( C (1+\|\nabla \mu \|)^\frac12  \log^\frac12(e+C(1+\| \nabla \mu\|))+ C\Big)^\frac45\\
&\leq \frac{\eta}{56} \| \partial_t \vp\|_{H^1(\Omega)}^2
+C (C+H)^\frac95 \log^\frac15(C+H) 
\Big( (C+H)^\frac14  \log^\frac12(C+H)\Big)^\frac45\\
&\leq \frac{\eta}{56} \| \partial_t \vp\|_{H^1(\Omega)}^2
+C (C+H)^2 \log^\frac35(C+H),
\end{align*}
whereas, using also \eqref{qH1}, we deduce
\begin{align}
I_{45}&= C \| \partial_t \vp\|_{\ast}^\frac14 
\| \partial_t \vp\|_{H^1(\Omega)}^\frac34  
\| \nabla \mu\|^\frac12 \| \uu\|
 \big( 1 + \| \vp\|_{H^2(\Omega)}^\frac12 \big)
\| q\|_{L^\infty(\Omega)}^\frac12 \notag \\
&\leq  \frac{\eta}{112} \| \partial_t \vp\|_{H^1(\Omega)}^2+
C \| \partial_t \vp\|_{\ast}^\frac25 \| \nabla \mu\|^\frac45 \|\uu \|^\frac{8}{5}(1+\| \nabla \mu\|)^\frac25 
\Big( \|q\|_{H^1(\Omega)}\log^\frac12(e+ \|q\|_{H^2(\Omega)}) + C \Big)^\frac45\notag\\
&\leq \frac{\eta}{112} \| \partial_t \vp\|_{H^1(\Omega)}^2+
C (C+H)^\frac15 \log^\frac15(C+H) (C+H)^\frac{7}{5}\notag\\
&\quad \times
\Big( (1+\| \nabla \mu\|)\log^\frac12(C+H) \log^\frac12(e+\| q\|_{H^2(\Omega)})+C\Big)^\frac45\notag\\
&\leq \frac{\eta}{112} \| \partial_t \vp\|_{H^1(\Omega)}^2+
C (C+H)^\frac85 \log^\frac15(C+H)
\Big( (C+H)^\frac12\log^\frac12(C+H) \log^\frac12(e+\| q\|_{H^2(\Omega)})\Big)^\frac45\notag\\
&\leq \frac{\eta}{112} \| \partial_t \vp\|_{H^1(\Omega)}^2+
C (C+H)^2 \log^\frac35(C+H) \log^\frac25(e+\| q\|_{H^2(\Omega)}).
\label{I44-1}
\end{align}
To control the last term, it is sufficient to perform a rough estimate (in terms of exponents)
since this is just a term in the logarithmic correction. We first observe that,
by \eqref{Ad2}, \eqref{EE} and \eqref{H2},
\begin{align*}
\| \partial_t \vp\|_{\ast}
&\leq C (1+\| \nabla \mu\|+\| \uu\| \|\vp \|_{L^\infty(\Omega)})\\
&\leq  C (1+\| \nabla \mu\|+\| \uu\| \|\vp \|_{H^2(\Omega)}^\frac12)\\
&\leq C (1+\| \nabla \mu\|+ \| \uu \| +\| \uu\| \|\nabla \mu \|^\frac14 ).
\end{align*} 
Then, going back to \eqref{qH2} and using the estimates \eqref{W2p} and \eqref{qH1}, it is not difficult to arrive at 
\begin{align}
\| q\|_{H^2(\Omega)}
&\leq C \| \vp\|_{W^{1,3}(\Omega)}\| \vp\|_{W^{2,3}(\Omega)} \| \nabla q\|+
C(1+\| \vp\|_{W^{1,3}(\Omega)})^2 \|\vp \|_{W^{2,3}(\Omega)} \| \nabla \mu\| \notag \\
&\quad + C \| \vp\|_{H^2(\Omega)} 
\| \partial_t \vp\|_{\ast}^\frac12 \| \partial_t \vp\|_{H^1(\Omega)}^\frac12 +C \notag \\
&\leq C(1+\| \nabla \mu\|)^4+ 
C(1+\|\nabla \mu\|+\| \uu\|)^2 \| \partial_t \vp\|_{H^1(\Omega)}^\frac12 \notag \\
&\leq C(C+H)^2+ C(C+H)\| \partial_t \vp\|_{H^1(\Omega)}^\frac12.
\label{qH2-2}
\end{align}
Combining \eqref{I44-1} with \eqref{qH2-2}, we find (here $C$ also changes from line to line to adjust exponent in the logarithm)
\begin{align*}
I_{45}& \leq \frac{\eta}{112} \| \partial_t \vp\|_{H^1(\Omega)}^2+
C (C+H)^2 \log^\frac35 (C+H) \log^\frac25 \Big(e+(C+H)^2+C(C+H)\| \partial_t \vp\|_{H^1(\Omega)}^\frac12 \Big)\\
&\leq \frac{\eta}{112} \| \partial_t \vp\|_{H^1(\Omega)}^2+
C (C+H)^2  \log \Big( C\big(e+(C+H)^2+ \| \partial_t \vp\|_{H^1(\Omega)}\big) \Big)\\
&\leq \frac{\eta}{112} \| \partial_t \vp\|_{H^1(\Omega)}^2+
C (1+H)^2  \log \Big(  (C+H)^2 \times \big(e+ \| \partial_t \vp\|_{H^1(\Omega)}\big) \Big)\\
&\leq \frac{\eta}{112} \| \partial_t \vp\|_{H^1(\Omega)}^2+
C (C+H)^2  \log \Big(  (C+ H)^2 \Big) + 
C(C+H)^2 \log \Big( e+ \|\partial_t \vp \|_{H^1(\Omega)} \Big) \\
&\leq \frac{\eta}{112} \| \partial_t \vp\|_{H^1(\Omega)}^2+
C (C+H)^2  \log (  C+ H ) + 
C(C+H)^2 \log \Big( e+ \|\partial_t \vp \|_{H^1(\Omega)} \Big).
\end{align*}
In order to handle the last term on the right-hand side above, we recall the following basic inequality 
(see also \cite[pag.~115]{FMT} for a similar inequality)
$$
x^2\log(e+y)\leq \varepsilon (e+y)^2 + x^2\log \Big(\frac{x}{\sqrt{2\varepsilon}}\Big), \quad \forall \, \varepsilon>0, x>0,y>0.
$$
Using this estimate with $\varepsilon=\frac{\eta}{224}$, $x=C+H$ and $y=\|\partial_t \vp \|_{H^1(\Omega)}$, we arrive at 
\begin{align*}
I_{45}& \leq \frac{\eta}{112} \| \partial_t \vp\|_{H^1(\Omega)}^2+ \frac{\eta}{224}(e+ \| \partial_t \vp\|_{H^1(\Omega)})^2+
C(C+H)^2\log\Big(e+ \frac{C+H}{\sqrt{\eta/112}}\Big)\\
&\leq \frac{\eta}{56} \| \partial_t \vp\|_{H^1(\Omega)}^2+
C(C+H)^2\log(C+ H).
\end{align*}
Next, we recall that the last terms we need to control are 
\begin{align*}
& I_{46}+I_{47} = C \| \partial_t \vp\|_{\ast}^\frac14 
\| \partial_t \vp\|_{H^1(\Omega)}^\frac34  
\| \nabla \mu\|^\frac12 \| \uu\| \\
& \quad\quad\quad\quad
\times\Big( \| \vp\|_{L^\infty(\Omega)} \| \nabla \vp\|_{L^\infty(\Omega)} \|\nabla \mu \|
+ \| \nabla \vp\|_{L^\infty(\Omega)} \| \nabla q\|\Big)^\frac12.
\end{align*}
Then, by \eqref{EE}, \eqref{BW2}, \eqref{vpt} and \eqref{qH1} we have
\begin{align*}
 I_{46}+I_{47}& \leq C (C+H)^\frac18 \log^{\frac18} (C+H) 
\| \partial_t \vp\|_{H^1(\Omega)}^\frac34 (C+H)^\frac34 \\
& \quad\quad \times \Big( \log^\frac12(C+H) \| \vp\|_{H^2(\Omega)}
\log^\frac12 (C+H) \| \nabla \mu\|\\ 
& \quad\quad\quad\quad
 + \| \vp\|_{H^2(\Omega)}\log^\frac12(C+H) 
 \big(1 + \|\nabla \mu \| \big)  \log^\frac12 (C+H) \Big)^\frac12 \\
&\leq C \| \partial_t \vp\|_{H^1(\Omega)}^\frac34 (C+H)^\frac78  \log^{\frac18} (C+H) 
   \Big( (C+H)^\frac34 \log(C+H) \Big)^\frac12\\
&\leq C\| \partial_t \vp\|_{H^1(\Omega)}^\frac34 (C+H)^\frac{5}{4}  \log^{\frac58} (C+H) \\
&\leq \frac{\eta}{28} \| \partial_t \vp\|_{H^1(\Omega)}^2+ C (C+H)^2\log(C+H).
\end{align*}
Now, going back to \eqref{test4} and collecting all the above estimates, we infer that
$$
\ddt H + \frac{\eta}{2} \|\partial_t \vp \|_{H^1(\Omega)}^2 \leq C(C+H)^2\log(C+H).
$$
Since $H\in L^1(0,T)$ (cf. \eqref{EE}), by applying Lemma \ref{GL2} we deduce the following double exponential estimate
$$
H(t)\leq (C+H(0))^{\mathrm{e}^{\int_0^t C (C+  H(s)) \, \d s}}, \quad \forall \, t \in [0,T].
$$
Recalling Assumption $A.3$ and the subsequent Remark~\ref{rem:H}, we notice that the value of $H$ is finite at the initial time. Hence, we obtain
the following bounds
$$
\| \nabla \mu \|_{L^\infty(0,T;L^2(\Omega))}+ \|\uu \|_{L^\infty(0,T;L^2(\Omega))}
 + \| \partial_t \vp\|_{L^2(0,T;H^1(\Omega))}
\leq C.
$$
Now, thanks to estimates \eqref{W2p}, \eqref{vpt} and \eqref{qH1}, we also infer that
$$
\| \vp \|_{L^\infty(0,T;W^{2,p}(\Omega))}+ \|\partial_t \vp\|_{L^\infty(0,T;(H^1(\Omega))')}
+ \| q\|_{L^\infty(0,T;H_0^1(\Omega))}
\leq C,
$$
for any $2\leq p<\infty$.
Thanks to \cite[Lemma 7.4]{GGW2018} (see also \cite[Theorem 2.2]{G2020}), we can deduce that 
\begin{equation}
\label{g:14}
\| F''(\vp)\|_{L^\infty(0,T;L^p(\Omega))}\leq C,
\end{equation}
for any $2\leq p<\infty$. 
Next, by \eqref{qH2-2} we now have
\begin{equation}
\label{g:13}
 \| q \|_{H^2(\Omega)} \leq C(1+ \| \partial_t \vp\|_{H^1(\Omega)}^{\frac12}),
\end{equation}
whence
$$
\| q \|_{L^4(0,T;H^2(\Omega))}\leq C.
$$
In addition to that, by \eqref{muH2} and \eqref{g:13}, we infer 
$$
 \| \mu\|_{H^2(\Omega)} \leq C(1+ \| \partial_t \vp\| + \| q \|_{L^\infty(\Omega)} )
 \leq C(1+ \| \partial_t \vp\|_{H^1(\Omega)}^{\frac12}), 
$$
which implies that 
$$
\| \mu\|_{L^4(0,T;H^2(\Omega))}\leq C.
$$
By exploiting \eqref{CHHS}$_1$, and the Sobolev embeddings,  we also deduce that 
\begin{align}
\| \uu\|_{H^1(\Omega)}& \leq \| q\|_{H^2(\Omega)}+ \| \vp \nabla \mu \|_{H^1(\Omega)} \notag \\
&\leq \| q\|_{H^2(\Omega)} + C \| \vp\|_{L^\infty(\Omega)} \| \mu \|_{H^2(\Omega)} + C \| \vp\|_{W^{1,\infty}(\Omega)} \| \nabla \mu\| \notag\\
&\leq C+ \| q\|_{H^2(\Omega)} + C \| \mu \|_{H^2(\Omega)}.
\label{uH1}
\end{align}
Thanks to the above regularity, we have
$$
\| \uu\|_{L^4(0,T;H^1(\Omega))}\leq C.
$$


\subsection{Uniqueness of Strong Solutions in Two Dimensions}

Let us consider a pair of strong solutions $(\uu_1,q_1,\vp_1)$ and $(\uu_2,q_2,\vp_2)$ 
originating from the same initial condition $\vp_0$. We define
$$
 \uu=\uu_1-\uu_2, \quad q=q_1-q_2, \quad \vp=\vp_1-\vp_2, \quad \mu=\mu_1-\mu_2,
$$
which solve
\begin{equation}
\label{diff-p}
\begin{cases}
 \uu + \nabla  q= -\vp_1 \nabla \mu- \vp \nabla \mu_2,\\
\div \uu= S,\\
\partial_t \vp + \div (\vp_1 \uu ) + \div ( \vp\uu_2 ) = \Delta  \mu + S,\\
\mu= -\Delta \vp+ \Psi'(\vp_1)-\Psi'(\vp_2),
\end{cases}
\end{equation}
where $ S= -m   \vp+ h(\vp_1)-h(\vp_2)$.
We first observe that 
\begin{equation}
\label{q-diff}
-\Delta q= S +\div(\vp_1 \nabla  \mu)+\div( \vp \nabla \mu_2).
\end{equation}
Multiplying \eqref{diff-p}$_1$ by $ \uu$, \eqref{diff-p}$_3$ by $ \mu$, \eqref{diff-p}$_4$ by $\partial_t  \vp$ and \eqref{q-diff} by $\varepsilon (-\Delta)^{-1}  q$ for some $\varepsilon \in (0,1)$ that will be chosen later, integrating over $\Omega$ and summing the resulting equations, we find
\begin{align*}
\ddt &\bigg[ \frac12 \| \nabla \vp\|^2 +\frac12 \int_{\Omega} L(\vp_1,\vp_2) | \vp|^2 \, \d x \bigg]+ \| \nabla \mu\|^2+ \|  \uu\|^2+ \varepsilon \| q\|^2 \\
&= \int_{\Omega}   \vp  \, \uu_2 \cdot  \nabla   \mu \, \d x + \int_{\Omega}  S \, \mu \, \d x+ \int_{\Omega}   S \,  q \, \d x - \int_{\Omega}  \vp \, \nabla \mu_2 \cdot   \uu \, \d x\\
&\quad + \int_{\Omega} \partial_t L(\vp_1,\vp_2) \frac{|  \vp|^2}{2} \,\d x+ \theta_0 \int_{\Omega}  \vp \partial_t  \vp \, \d x 
+ \varepsilon\int_{\Omega} S (-\Delta )^{-1} q \, \d x \\
&\quad - \varepsilon \int_{\Omega} \vp_1 \nabla \mu \cdot \nabla (-\Delta)^{-1}  q \, \d x- \varepsilon \int_{\Omega}  \vp \, \nabla \mu_2 \cdot \nabla (-\Delta)^{-1}  q \, \d x,
\end{align*}
where
$$
L(\vp_1,\vp_2)= \int_0^1 F''(\tau \vp_1 +(1-\tau )\vp_2) \, \d \tau \geq \theta>0.
$$
We now estimate all the terms on the right-hand side of the above equality. In the sequel, we will use the notation 
$$
Y= \frac12 \| \nabla \vp\|^2 +\frac12 \int_{\Omega} L(\vp_1,\vp_2) | \vp|^2 \, \d x,
$$
and we will repeatedly use that
$
\|\vp\|_{H^1(\Omega)}^2 \leq C Y 
$, for some positive constant $C$.
By using the Sobolev embeddings, we have
\begin{align*}
 \int_{\Omega}   \vp  \, \uu_2 \cdot  \nabla   \mu \, \d x 
& \leq \| \vp\|_{L^6(\Omega)} \| \uu_2 \|_{L^3(\Omega)} \| \nabla \mu\|\\
&\leq \frac{1}{6} \|\nabla \mu\|^2 +C \| \uu_2\|_{L^3(\Omega)}^2 Y,
\end{align*}
and
\begin{align*}
- \int_{\Omega}  \vp \, \nabla \mu_2 \cdot   \uu \, \d x &\leq \| \vp\|_{L^6(\Omega)} \| \nabla \mu_2\|_{L^3(\Omega)} \| \uu\|\\
&\leq \frac14 \| \uu\|^2 + C \| \nabla \mu_2\|_{L^3(\Omega)}^2 Y,
\end{align*}
By using \eqref{diff-p}$_4$, we obtain
\begin{align*}
 \int_{\Omega}  S \,  \mu \, \d x  &= \int_{\Omega} \big( -m \vp + h(\vp_1)-h(\vp_2) \big)\big(-\Delta \vp +F'(\vp_1)-F'(\vp_2)-\theta_0 \vp \big) 
\, \d x \\
&= \int_{\Omega} -m | \nabla \vp|^2 + \nabla (h(\vp_1)-h(\vp_2))\cdot \nabla \vp \, \d x \\
&\quad + \int_{\Omega}  \big( -m \vp + h(\vp_1)-h(\vp_2) \big) \big(F'(\vp_1)-F'(\vp_2)-\theta_0 \vp \big) 
\, \d x \\
&\leq C \| \nabla \vp\|^2+ C \| \vp\|^2 + C \| \vp\| \big( 1 +  \| F''(\vp_1)\|_{L^3(\Omega)}
   + \| F''(\vp_2)\|_{L^3(\Omega)} \big) \| \vp\|_{L^6(\Omega)}\\
& \leq C Y.
\end{align*}
Here we have used \eqref{g:14} in the last inequality.
We also have
\begin{align*}
\int_{\Omega}   S \,  q \, \d x \leq C \| \vp\| \|q \|\leq \frac{\varepsilon}{4} \| q\|^2+CY.
\end{align*}
By definition of $L$, and observing that $|F'''(s)|\leq C (F''(s))^2$ and $F''$ is convex, we have
\begin{align*}
 & \int_{\Omega} \partial_t L(\vp_1,\vp_2) \frac{|  \vp|^2}{2} \,\d x\\
 & \quad \leq C \int_{\Omega} \int_0^1 \big( \tau F''(\vp_1)+(1-\tau) F''(\vp_2) \big)^2 
   \big| \tau \partial_t \vp_1 +(1-\tau) \partial_t \vp_2 \big| \d \tau \, \frac{|\vp|^2}{2} \, \d x\\
 & \quad \leq C \big( \| F''(\vp_1)\|_{L^6(\Omega)}^2+ \| F''(\vp_2)\|_{L^6(\Omega)}^2\big)\big( \|\partial_t \vp_1 \|_{L^3(\Omega)} 
  +\| \partial_t \vp_2\|_{L^3(\Omega)} \big) \| \vp\|_{L^6(\Omega)}^2\\
 & \quad \leq C \big( \|\partial_t \vp_1 \|_{L^3(\Omega)} +\| \partial_t \vp_2\|_{L^3(\Omega)} \big) Y.
\end{align*}
By using \eqref{diff-p}$_3$, the boundary condition \eqref{bc} and the Sobolev embeddings, we find
\begin{align*}
\theta_0 \int_{\Omega}  \vp \partial_t  \vp \, \d x 
  &= \theta_0 \int_{\Omega} \vp (\Delta \mu +S - \div( \vp_1 \uu + \vp \uu_2)) \, \d x \\
&= \theta_0 \int_{\Omega} -\nabla \vp \cdot \nabla \mu + S \vp + \nabla \vp \cdot (\vp_1 \uu+ \vp \uu_2 ) \, \d x \\
&\leq C \| \nabla \vp\| \| \nabla \mu\| + C \| \vp\|^2 + C \|\nabla \vp \| \| \vp_1\|_{L^\infty(\Omega)} \| \uu\|
  + C \| \nabla \vp\| \| \vp \|_{L^6(\Omega)} \| \uu_2\|_{L^3(\Omega)}\\
&\leq \frac16 \| \nabla \mu\|^2 + \frac14 \| \uu\|^2 + C (1+\| \uu_2\|_{L^3(\Omega)}) Y,
\end{align*}
By using the regularity theory of the Dirichlet problem, and recalling that $\varepsilon<1$, we have
\begin{align*}
\varepsilon \int_{\Omega} S (-\Delta )^{-1} q \, \d x \leq C \|\vp \| \| q\| \leq \frac{\varepsilon}{4} \| q\|^2+ C Y,
\end{align*}
\begin{align*}
- \varepsilon \int_{\Omega} \vp_1 \nabla \mu \cdot \nabla (-\Delta)^{-1}  q \, \d x& \leq \varepsilon C \|\vp_1\|_{L^\infty(\Omega)} \| \nabla \mu \| \| q\|\leq  \frac16 \| \nabla \mu\|^2 + \varepsilon^2 C \|q \|^2 ,
\end{align*}
and
\begin{align*}
-\varepsilon \int_{\Omega}  \vp \, \nabla \mu_2 \cdot \nabla (-\Delta)^{-1}  q \, \d x \leq \varepsilon C \|\vp \|_{L^6(\Omega)} \|\nabla \mu_2 \|_{L^3(\Omega)} \| q\|\leq \frac{\varepsilon}{4} \| q \|^2 + C  \| \nabla \mu_2\|^2 Y.
\end{align*}
Collecting all the above estimates together,  we eventually arrive at the differential inequality 
\begin{align*}
\ddt Y &+ \frac12 \| \nabla \mu\|^2+ \frac12 \|  \uu\|^2+ \Big(\varepsilon- \frac{3\varepsilon}{4}- \varepsilon^2 C \Big) \| q\|^2 \\
& \leq C \Big(1+  \| \uu_2\|_{L^3(\Omega)}^2+\| \nabla \mu_2\|_{L^3(\Omega)}^2 + \| \partial_t \vp_1\|_{L^3(\Omega)}+ \| \partial_t \vp_2\|_{L^3(\Omega)}   \Big) Y.
\end{align*}
By choosing $\varepsilon$ sufficiently small, we finally end up with
\begin{align*}
\ddt Y \leq C \Big(1+  \| \uu_2\|_{L^3(\Omega)}^2+\| \nabla \mu_2\|_{L^3(\Omega)}^2 + \| \partial_t \vp_1\|_{L^3(\Omega)}+ \| \partial_t \vp_2\|_{L^3(\Omega)}   \Big) Y.
\end{align*}
Thus, an application of the Gronwall lemma entails that $\vp_1(t)=\vp_2(t)$ for all $t \in [0,T]$. In turn, this immediately  implies that $\uu_1(t)=\uu_2(t)$ and $q_1(t)=q_2(t)$ for all $t \in [0,T]$.


\section{Local Existence of  Strong Solutions in Three Dimensions}
\label{S4}
\setcounter{equation}{0}

This section is devoted to the analysis of the strong solutions to system \eqref{CHHS}-\eqref{bc} in the three dimensional setting. 

\begin{theorem}\label{teo:3d}
Let $\Omega$ be a bounded domain with smooth boundary in $\mathbb{R}^3$. Assume the conditions~{\rm (A1)--(A3)} hold. 
Then, there exists a time $T_0>0$ and at least one strong 
solution $(\uu,q,\vp)$ to system \eqref{CHHS}-\eqref{bc} such that
\begin{align}
\label{regou3}
& \uu \in L^\infty(0,T_0; L^2(\Omega))\cap L^4(0,T_0;H^1(\Omega)),\\
\label{regoq3}
& q \in L^\infty(0,T_0;H_0^1(\Omega))\cap L^4(0,T_0;H^2(\Omega)),\\
\label{regofhi3}
& \vp \in L^\infty(0,T_0;W^{2,6}(\Omega))\cap H^1(0,T_0;H^1(\Omega)), \\
\label{regomu3}
& \mu \in L^\infty(0,T_0;H^1(\Omega)) \cap L^4(0,T_0;H^2(\Omega)),\\
\label{regopsi3}
& \Psi'(\fhi) \in L^\infty(0,T_0;L^6(\Omega)).
\end{align}
Such a strong solution satisfies the system \eqref{CHHS}-\eqref{bc} 
almost everywhere in $\Omega \times (0,T_0)$ 
and assumes the initial value $\vp(\cdot,0)=\vp_0 (\cdot)$.
\end{theorem}

\begin{remark}
The uniqueness of the strong solutions obtained in Theorem \ref{teo:3d} remains an open issue. The argument used in the two dimensional case cannot be applied due 
to the lack of regularity for the derivatives of the potential (i.e. $\Psi''(\fhi)$ and $\Psi'''(\fhi)$) in three dimensions. On the other hand, the control 
of the difference of two solutions in weaker norms as in \cite{G2020,GGW2018} does not seem to be possible here due to the  boundary conditions \eqref{bc}. 
\end{remark}

\begin{proof}[Proof of Theorem \ref{teo:3d}]
We first observe that the basic {\it a priori} estimates performed 
in the Subsection ~\ref{subsec:en}, i.e. the Total Mass Dynamics and the Energy Estimates, are also valid in the three dimensional case with no variation in the proof. 
As a consequence, we still achieve \eqref{EE},
with no restriction on the final time $T$. Similarly, we report that 
\begin{equation}
\label{S-H1-3}
\|S \|_{L^\infty(0,T;H^1(\Omega))}\leq C,
\end{equation}
and 
\begin{equation}
\label{muH1-3}
\| \mu\|_{H^1(\Omega)}\leq C(1+\| \nabla \mu\|).
\end{equation}
Exploiting once again \cite[Theorem 5.1]{G2020} (cf. \cite[Lemmas 7.3 and 7.4]{GGW2018}), we have 
\begin{equation}
\label{H2-3}
\| \vp\|_{H^2(\Omega)}^2 \leq C(1+ \| \nabla \mu\|), \quad 
\| \vp\|_{W^{2,6}(\Omega)}\leq C(1+\| \nabla \mu\|).
\end{equation}
We now proceed with the higher order estimates. We notice that the validity of the relation \eqref{test3} is independent of the dimension.  In particular, for the reader's convenience we report that 
\begin{align}
 \ddt \bigg[ \frac12 \| \nabla \mu\|^2 &+ \frac12 \| \uu\|^2 - (S,\mu)- (S,q) \bigg] +
 \eta \|\partial_t \vp\|_{H^1(\Omega)}^2 + \int_{\Omega} F''(\vp) |\partial_t \vp |^2 \, \d x \notag \\
 & \leq C \| \partial_t \fhi\|_{\ast}^2+ \theta_0 \int_{\Omega} |\partial_t \vp|^2 \, \d x- (\partial_t S, \mu)- (\partial_t S, q)-
\int_{\Omega} \partial_t \vp \nabla \mu \cdot \uu \, \d x \nonumber \\
 & \leq C \| \partial_t \fhi\|^2 - (\partial_t S, \mu) - (\partial_t S, q)
 - \int_{\Omega} \partial_t \vp \nabla \mu \cdot \uu \, \d x \nonumber  \\
 & =: I_1 + I_2 + I_3 + I_4,
 \label{g:21}
\end{align}
for a positive constant $\eta$.
By the elliptic regularity of the system \eqref{Smu2}, combined with Sobolev's embeddings and the estimates \eqref{EE}, it follows that
\begin{equation}
\label{g:25}
  \| q \| 
   \le C \| S \| + C \| \fhi \nabla \mu \|_{L^{\frac65}(\Omega)}
   \le C( 1+  \| \fhi \|_{H^1(\Omega)} \| \nabla \mu \| )
   \le C ( 1 + \| \nabla \mu \| ).
\end{equation}
Then, we easily infer from \eqref{S-H1-3} and \eqref{g:25} that the estimates \eqref{Smu} and \eqref{Sq} remain true in the three dimensional setting. Hence, as before, we can define the functional $H$ as in \eqref{defiH} and notice that relation \eqref{H} still holds.
Moreover, thanks to \eqref{Ad3} and \eqref{H2-3}, we deduce that
\begin{align}
\nonumber
  \| \nabla q \| 
   & \le C \| S \| + C \| \fhi \nabla \mu \|
   \le C + C \| \fhi \|_{L^\infty(\Omega)} \| \nabla \mu \|\\
 \label{g:25b}
  & \le C + C \| \fhi \|_{H^1(\Omega)}^{\frac12} \| \fhi \|^{\frac12}_{H^2(\Omega)} \| \nabla \mu \|
   \le C (1 + \| \nabla \mu \|^{\frac54} ).
\end{align}
Now, we can control the first three terms on the right-hand side
of \eqref{g:21} as in Section \ref{S3}. Indeed, by \eqref{muH1-3}, \eqref{g:25} and the form of $S$, we notice that
\begin{equation}
 \label{g:24}
 I_2  = - (\partial_t S, \mu)  \le C \| \dt\fhi \|^2 + C (1 + \|\nabla \mu \|)^2,
\end{equation}
and
$$
  I_3 = - (\partial_t S, q) \le C \| \dt\fhi \|^2 + C (1 + \|\nabla \mu \|)^2.
$$
Collecting the above computations, we readily arrive at
\begin{equation}
\label{g:26}
  I_1 + I_2 + I_3 \le C \| \dt \fhi \|^2 + C (1 + \|\nabla \mu \|)^2 
 \le C \| \dt \fhi \|^2 + C (1 + H).
\end{equation}
Now, using \eqref{CHHS}$_3$, \eqref{Ad3} and \eqref{H2-3}, we have
\begin{align}
\nonumber
  \| \dt \fhi \|_{\ast} &
  \le  \| \nabla \mu \| + \| \uu\| \| \fhi \|_{L^\infty(\Omega)}
   +C \| S \|\\
 \nonumber
   & \le \| \nabla \mu \| + C \| \uu \| \| \fhi \|_{H^1(\Omega)}^{\frac12} \| \fhi \|_{H^2(\Omega)}^{\frac12} + C\\
 \label{g:27}
  & \le C(1+\| \nabla \mu \|)+ C \| \uu \| ( 1 + \| \nabla \mu \|)^\frac14
  \le C (C+ H)^{\frac58}.
\end{align}
Consequently, we find
\begin{align}
  \| \dt \fhi \|
   & \le  C \| \dt \fhi \|_{\ast}^{\frac12} \| \dt \fhi \|_{H^1(\Omega)}^{\frac12} 
     \le C (C + H)^{\frac{5}{16}} \| \dt \fhi \|_{H^1(\Omega)}^{\frac12}.
 \label{g:27b}
\end{align}
Replacing the above into \eqref{g:26}, we obtain
\begin{equation}
\label{g:26b}
  I_1 + I_2 + I_3 \le \frac{\eta}{4} \|  \dt \fhi \|_{H^1(\Omega)}^2 
  + C ( 1 + H )^{\frac54}.
\end{equation}
Next, we recall the Gagliardo-Nirenberg inequality (cf. \eqref{GN3} with $q=6$)
\begin{equation}
\label{g:GN}
  \| f\|_{L^\infty(\Omega)}
   \le C \| f \|^{\frac14} \| f \|_{W^{1,6}(\Omega)}^{\frac34},  \quad \forall\, f \in W^{1,6}(\Omega).
\end{equation}
Proceeding as in \eqref{g:12} and using \eqref{muH1-3}, 
\eqref{g:27b}, and \eqref{g:GN}, we deduce that
\begin{align}
 \nonumber
  \| \mu \|_{H^2(\Omega)} 
  & \le C | \ov{\mu} | + C \| \Delta \mu \| +
   C \Big\| \frac{\vp}{1+\vp^2} q \Big\|_{H^2(\Omega)}\\
 \nonumber
  & \le C ( 1 + \| \nabla \mu \| )
   + C \big( \| \dt\fhi \| + \| S (\fhi - 1) \| + \| \uu \cdot \nabla \fhi \| \big)
   + C \Big\| \frac{\vp}{1+\vp^2} q \Big\|_{H^2(\Omega)}\\
 \nonumber
  & \le C ( C + H )^{\frac12}
   + C \big( 1+ \| \dt\fhi \|+ \| \uu \| \| \nabla \fhi \|_{L^\infty(\Omega)} \big)
   + C \Big\| \frac{\vp}{1+\vp^2} q \Big\|_{H^2(\Omega)}\\
 \nonumber
  & \le C ( C + H )^{\frac12}
   + C  ( C + H )^\frac{5}{16}\| \dt \fhi \|_{H^1(\Omega)}^\frac12 
   + \| \uu \| \| \nabla \fhi \|^{\frac14} \| \fhi \|_{W^{2,6}(\Omega)}^{\frac34} 
   + C \Big\| \frac{\vp}{1+\vp^2} q \Big\|_{H^2(\Omega)}\\
 \nonumber
  & \le C ( C + H )^{\frac12}
   + C ( C + H )^\frac{5}{16}\| \dt \fhi \|_{H^1(\Omega)}^\frac12  
 + \| \uu \| ( 1 + \| \nabla \mu \|^{3/4} ) 
   + C \Big\| \frac{\vp}{1+\vp^2} q \Big\|_{H^2(\Omega)}\\
 \label{g:41}
  & \le C ( C + H )^{\frac78}
   + C ( C + H )^\frac{5}{16}\| \dt \fhi \|_{H^1(\Omega)}^\frac12   + C \Big\| \frac{\vp}{1+\vp^2} q \Big\|_{H^2(\Omega)}.
\end{align}
In order to provide a control of the last term in \eqref{g:41}, we report the analogue of \eqref{gfhi} in three dimensions
\begin{align}
\Big\| \frac{\vp}{1+\vp^2}\Big\|_{H^2(\Omega)}
\leq C \big( 1+ \| \nabla \vp\|_{L^4(\Omega)}^2+ \|\vp\|_{H^2(\Omega)} \big)
\leq C\big(1+ \| \vp\|_{H^2(\Omega)}^\frac32\big).
\label{gfhi-3}
\end{align}
Here we have used \eqref{L3}.
Next, recalling \eqref{ell:q} and \eqref{ell:q2} and using \eqref{gfhi-3}, we find 
\begin{align}
\nonumber
\Big\|\frac{\vp}{1+\vp^2} q\Big\|_{H^2(\Omega)}
&\leq C+ C \| \vp\|_{L^\infty(\Omega)} \| \nabla \vp\|_{L^\infty(\Omega)} \|\nabla \mu \|
+ C \| \nabla \vp\|_{L^\infty(\Omega)} \| \nabla q\|+
C \| \partial_t \vp\| \\
&\quad + C(1+\|\vp \|_{H^2(\Omega)}^\frac32) \| q\|_{L^\infty(\Omega)}.
\label{phiqH2-3}
\end{align}
To control the $L^\infty$-norm of $q$, we report \eqref{qH2} for the reader's convenience
\begin{align}
\| q\|_{H^2(\Omega)}
&\leq C+ C \| \vp\|_{L^\infty(\Omega)} \| \nabla \vp\|_{L^\infty(\Omega)} \|\nabla q \|+ C (1+\| \vp\|_{L^\infty(\Omega)}^2)\| \nabla \vp\|_{L^\infty(\Omega)} \| \nabla \mu\| \notag \\
&\quad+ C \| \vp\|_{L^\infty(\Omega)} \|\partial_t \vp \|.
\label{qH2-3}
\end{align}
Exploiting \eqref{Ad3}, \eqref{EE}, \eqref{H2-3}, \eqref{g:25b}, \eqref{g:27b} and \eqref{g:GN}, we obtain
\begin{align}
 \nonumber
 \| q \|_{H^2(\Omega)} 
 & \le C+ C \| \fhi \|_{H^1(\Omega)}^{\frac12} \| \fhi \|_{H^2(\Omega)}^{\frac12}  \| \nabla \fhi \|^{\frac14} \| \fhi \|_{W^{2,6}(\Omega)}^{\frac34}  \| \nabla q \| \\ 
 \nonumber
 & \quad
   + C \big(1 + \| \fhi \|_{H^1(\Omega)} \| \fhi \|_{H^2(\Omega)} \big)   
     \| \nabla \fhi \|^{\frac14} \| \fhi \|_{W^{2,6}(\Omega)}^{\frac34}  \| \nabla \mu \| 
    + C \| \fhi \|_{H^1(\Omega)}^{\frac12} \| \fhi \|_{H^2(\Omega)}^{\frac12} \| \partial_t \vp \|   \\
 \nonumber
 & \le C+ C \| \fhi \|_{H^2(\Omega)}^{\frac12} \| \fhi \|_{W^{2,6}(\Omega)}^{\frac34} \| \nabla q \|   + C \big(1 + \| \fhi \|_{H^2(\Omega)} \big)   
    \|  \fhi \|_{W^{2,6}(\Omega)}^{\frac34} \| \nabla \mu \| 
    + C \| \fhi \|_{H^2(\Omega)}^{\frac12} \| \partial_t \vp \| \\
 \nonumber
 & \le C (1 + \| \nabla \mu \| )^\frac14      
    ( 1 + \| \nabla \mu  \| )^\frac34 (1 + \| \nabla \mu \| )^{\frac54}
   + C (1 + \| \nabla \mu \| )^\frac12   
     (1 + \| \nabla \mu \|)^{\frac34}  \| \nabla \mu \| \\
  \nonumber
 & \quad\quad\quad
    + C ( 1 + \| \nabla \mu \| )^\frac14 (C + H)^{\frac{5}{16}} \|  \dt \fhi \|_{H^1(\Omega)}^\frac12  \\
 \label{g:42}
  & \le C (C+H)^{\frac98} + C  (C+H)^{\frac{7}{16}} \| \dt \fhi \|_{H^1(\Omega)}^\frac12.
\end{align}
Now, we go back to \eqref{phiqH2-3}. By similar computations, and using \eqref{g:42}, we have
\begin{align}
\nonumber
\Big\|\frac{\vp}{1+\vp^2} q\Big\|_{H^2(\Omega)}
&\leq C+ C \| \vp\|_{H^1(\Omega)}^\frac12 \| \vp\|_{H^2(\Omega)}^\frac12 
\| \nabla \vp\|^\frac14 \| \vp\|_{W^{2,6}(\Omega)}^\frac34 \|\nabla \mu \|
+ C \| \nabla \vp\|^\frac14 \| \vp\|_{W^{2,6}(\Omega)}^\frac34 
\| \nabla q\| \notag \\
&\quad + C\| \partial_t \vp\| + C(1+\|\vp \|_{H^2(\Omega)}^\frac32) \| q\|_{H^1(\Omega)}^\frac12 \| q\|_{H^2(\Omega)}^\frac12\notag \\
&\leq C + C  \| \vp\|_{H^2(\Omega)}^\frac12 \| \vp\|_{W^{2,6}(\Omega)}^\frac34 \|\nabla \mu \|+ C \| \vp\|_{W^{2,6}(\Omega)}^\frac34 
\| \nabla q\| \notag \\
&\quad + C(1+\|\vp \|_{H^2(\Omega)}^\frac32) \| q\|_{H^1(\Omega)}^\frac12 \| q\|_{H^2(\Omega)}^\frac12
 + C \| \partial_t \vp\| \notag \\
&\leq C(1+\| \nabla \mu\|)^\frac14 (1+\| \nabla \mu\|)^\frac34(1+\|\nabla \mu \| )+ C (1+\| \nabla \mu\|)^\frac34 (1+\| \nabla \mu\|)^\frac54 \notag \\
&\quad +C (1+\| \nabla \mu\|)^\frac34 (1+\| \nabla \mu\|)^\frac54 \big( (C+H)^{\frac{9}{16}} + (C+H)^{\frac{7}{32}}  \|  \dt \fhi \|_{H^1(\Omega)}^\frac14 \big)\notag \\
& \quad +C (C + H)^{\frac{5}{16}} \| \dt \fhi \|_{H^1(\Omega)}^{\frac12} \notag \\
&\leq C(C+H)+ C(C+H)^\frac{25}{16}
+C(C+H)^\frac{39}{32}\| \partial_t \fhi\|_{H^1(\Omega)}^\frac14
+C(C+H)^\frac{5}{16} \| \partial_t \fhi \|_{H^1(\Omega)}^\frac12 \notag \\
&\leq C(C+H)^\frac{25}{16}+
C(C+H)^\frac{39}{32}\| \partial_t \fhi\|_{H^1(\Omega)}^\frac14
+C(C+H)^\frac{5}{16} \| \partial_t \fhi \|_{H^1(\Omega)}^\frac12.
\label{phiqH2-3+}
\end{align}
Replacing \eqref{phiqH2-3+} into \eqref{g:41} we finally have
\begin{equation}
\label{g:44}
  \| \mu \|_{H^2(\Omega)} 
   \le  C(C+H)^\frac{25}{16}+ C ( C + H )^\frac{5}{16}\| \dt \fhi \|_{H^1(\Omega)}^\frac12
+C(C+H)^\frac{39}{32}\| \partial_t \fhi\|_{H^1(\Omega)}^\frac14.
\end{equation}
We are now ready to provide a bound for the term $I_4$. Using the 
above relations, we deduce that
\begin{align}
  I_4 & = - \int_{\Omega} \partial_t \vp \nabla \mu \cdot \uu \, \d x  
  \nonumber \\
   & \le \| \dt\fhi \|_{L^6(\Omega)} \| \uu \| \| \nabla \mu \|_{L^3(\Omega)}
  \nonumber \\
   &\le C \| \dt\fhi \|_{H^1(\Omega)} \| \uu \| \| \nabla \mu \|^{\frac12} \| \mu \|^{\frac12}_{H^2(\Omega)}\nonumber \\
 \nonumber
   & \le C \| \dt\fhi \|_{H^1(\Omega)} ( C + H )^{\frac34} \| \mu \|_{H^2(\Omega)} ^{\frac12}\\
 \nonumber
   & \le C \| \dt\fhi \|_{H^1(\Omega)} ( C+ H )^{\frac34} 
      \Big[ C(C+H)^\frac{25}{32}+ C ( C + H )^\frac{5}{32}\| \dt \fhi \|_{H^1(\Omega)}^\frac14
+C(C+H)^\frac{39}{64}\| \partial_t \fhi\|_{H^1(\Omega)}^\frac18 \Big] \nonumber \\
 \nonumber
   & \le  C \| \dt\fhi \|_{H^1(\Omega)} (C+H)^{\frac{49}{32}} 
    + C \|  \dt \fhi \|_{H^1(\Omega)}^{\frac54} (C+H)^{\frac{29}{32}}
    + C \|  \dt \fhi \|_{H^1(\Omega)}^{\frac98} (C+H)^{\frac{87}{64}}\\
 \label{g:31}
   & \le \frac{\eta}{4} \| \dt\fhi \|_{H^1(\Omega)}^2
    + C (C+H)^{\frac{87}{28}}. 
\end{align}
Replacing \eqref{g:26b} and \eqref{g:31} into \eqref{g:21}, we arrive at
\begin{equation}
 \ddt H  + \frac{\eta}{2} \| \partial_t \vp\|_{H^1(\Omega)}^2
  \le C (C+H)^{\frac{87}{28}}.
 \label{g:21x}
\end{equation}
Thus, by comparison principle for ODE's, we obtain that there exists
a time $T_0>0$ depending in particular on the value of $H$ at time $t=0$
such that 
\begin{equation}
  \| \nabla \mu \|_{L^\infty(0,T_0;L^2(\Omega))} + \| \uu\|_{L^\infty(0,T_0;L^2(\Omega))} + \| \vp \|_{H^1(0,T_0;H^1(\Omega))} \le C.
 \label{st:g1}
\end{equation}
Thanks to \eqref{muH1-3}, \eqref{H2-3}, \eqref{g:25} and \eqref{g:25b}, this immediately implies that
$$
 \|  \mu \|_{L^\infty(0,T_0;H^1(\Omega))} 
 + \| \vp\|_{L^\infty(0,T_0;W^{2,6}(\Omega))}
+\| q\|_{L^\infty(0,T_0;H_0^1(\Omega))}\leq C.
$$
By comparison in \eqref{CHHS}$_4$, it easily follows \eqref{regopsi3}.
To get further regularity of $q$ and $\mu$, we then notice that, by \eqref{regofhi3} and Sobolev's embeddings, both $\fhi$ and $\nabla \fhi$ are uniformly bounded. 
Hence, recalling \eqref{g:42} and \eqref{g:44}, we also infer that 
$$
\| q\|_{L^4(0,T_0;H^2(\Omega))}\leq C,\quad 
\| \mu\|_{L^4(0,T_0;H^2(\Omega))}\leq C.
$$ 
Finally, observing that \eqref{uH1} holds in three dimensions, the above regularities entail the second of \eqref{regou3}.
\end{proof}


\section*{Acknowledgments}
\noindent
This research was supported by the Italian Ministry of Education, University and Research
(MIUR): Dipartimenti di Eccellenza Program (2018-2022),
Department of Mathematics ``F.~Casorati'', University of Pavia.
In addition, this research has been performed in the framework of the by GNAMPA-INdAM Project 
``Analisi matematica di modelli a interfaccia diffusa per fluidi complessi" and of the project Fondazione Cariplo-Regione Lombardia  MEGAsTAR
``Matema\-tica d'Eccellenza in biologia ed ingegneria come acceleratore
di una nuova strateGia per l'ATtRattivit\`a dell'ateneo pavese''. The present paper
also benefits from the support of the MIUR-PRIN Grant 2015PA5MP7 ``Calculus of Variations'' for GS, of the GNAMPA (Gruppo Nazionale per l'Analisi Matematica, la Probabilit\`a e le loro Applicazioni)
of INdAM (Istituto Nazionale di Alta Matematica) for AG, ER and~GS, and of a grant from the Research Grants Council of the Hong Kong Special Administrative Region (Project No.: HKBU 14302319) for KFL.


\end{document}